%% file: gdos.tex
\pdfoutput=1

 \documentclass{siamltex}		% Specifies the document style. 
\usepackage{graphicx,multirow,comment,amssymb,amsmath,amsfonts, nicefrac, 
color,mathtools,algorithmic,algorithm}

\input{preamble}

\newcommand{\nvec}{n_{\mathrm{vec}}}
\newcommand{\wt}{\widetilde}
\def\trans{^{\mathsf{T}}}
\def\invt{^{-\mathsf{T}}}%
\def\up#1{^{(#1)}}%
\providecommand{\norm}[1]{\Vert #1 \Vert} 
\def\inv{^{-1}}%
\title{Fast computation of spectral densities for generalized eigenvalue problems}
\date{today}
\author{Yuanzhe Xi 
\thanks{\{\texttt{yxi,saad}\}@umn.edu; 
Work supported in part (applications and practical aspects) by 
by the Scientific Discovery through Advanced Computing (SciDAC)
program funded by U.S.  Department of Energy, Office of Science,
Advanced Scientific Computing Research and Basic Energy Sciences
DE-SC0008877 and in part (theory) by NSF under grant 
CCF-1505970}
 \and Ruipeng Li 
\thanks{Center  for Applied  Scientific Computing,
    Lawrence  Livermore National  Laboratory,  P. O.  Box 808,  L-561,
    Livermore,   CA  94551   {(\tt{li50@llnl.gov})}.  This   work  was
    performed under the  auspices of the U.S. Department  of Energy by
    Lawrence    Livermore   National    Laboratory   under    Contract
    DE-AC52-07NA27344 (LLNL-JRNL-xxxxxx).} 
\and Yousef Saad
\footnotemark[1]}

\begin{document}

\maketitle

\begin{abstract}
  The distribution  of the eigenvalues  of a Hermitian  matrix (or of a Hermitian matrix pencil) reveals
  important features of the  underlying problem, whether a Hamiltonian
  system   in   physics,   or   a   social   network   in   behavioral
  sciences.  However,  computing  all the  eigenvalues  explicitly  is
  prohibitively  expensive for  real-world  applications.  This  paper
  presents two types  of methods to efficiently  estimate the spectral
  density  of a  matrix pencil  $(A,  B)$ when  both $A$  and $B$  are
  Hermitian and, in addition, $B$  is positive definite. The first one
  is based  on the Kernel  Polynomial Method  (KPM) and the  second on
  Gaussian  quadrature   by  the  Lanczos  procedure.    By  employing
  Chebyshev polynomial  approximation techniques, we can  avoid direct
  factorizations in  both methods, making the  resulting algorithms
  suitable for  large matrices. Under  some assumptions,
  we prove bounds that suggest that the Lanczos method  converges twice as fast as the KPM
  method. Numerical examples further  indicate that the Lanczos method
  can provide  more accurate  spectral densities when  the eigenvalue
  distribution is highly  non-uniform. As an application,  we show how
  to use the computed spectral  density to partition the spectrum into
  intervals that contain roughly the  same number of eigenvalues. 
  This procedure, which makes it possible to compute the spectrum by parts,
  is a key ingredient in the  new breed of  eigensolvers that exploit ``spectrum slicing".
\end{abstract}



\pagestyle{myheadings} \thispagestyle{plain}
\markboth{\uppercase{Yuanzhe Xi, Ruipeng Li, and Yousef Saad}}{\uppercase{Fast spectral density estimation}}

\section{Introduction} \label{sec:intro}
The problem of estimating the \emph{spectral density} 
of  an $n\times n$ Hermitian matrix $A$, has many applications in science
and engineering.  The spectral density is termed 
\emph{density of states} (DOS) in  solid state physics where it plays a key role.
Formally, the   DOS is defined as
\eq{eq:DOS0}
\phi(t) = 
\frac{1}{n}   
\sum_{j=1}^n \delta(t - \lambda_j) ,
\en 
where $\delta$ is the Dirac $\delta$-function or Dirac distribution,
and the $\lambda_j$'s are the eigenvalues of $A$, assumed here to be labeled
increasingly. In general, the formal definition of the spectral density as expressed by
\nref{eq:DOS0} is not  easy to use in practice.  Instead,  it is often
approximated, or more specifically smoothed,  and it is this resulting
approximation,  usually  a  smooth  function,  that  is  sought.  
%This
%function  can  be viewed  as  a  probability density  distribution  to
%measure the likelihood  of finding eigenvalues near any given  point on the real line.

Estimating spectral densities can be useful in a wide range of applications 
apart from the important ones in physics, chemistry and network analysis, see, e.g.,  
\cite{2015-siam-ns,kpmsurvey2006,DOSpaper16}.
One such application is the problem of  estimating 
the number $\eta_{[a, \ b]} $  of eigenvalues 
in an interval $[a,\ b]$. Indeed, this number can be obtained by
integrating the spectral density in the interval: 
\eq{eq:DOS1}
\eta_{[a,\ b]} = \int_a^b 
\sum_j  \delta(t - \lambda_j) \ dt \equiv 
\int_a^b n \phi(t) dt \ .
\en 
Thus, one can view $\phi(t)$ as a probability %
distribution function which gives the probability of finding 
eigenvalues of $A$ in a given infinitesimal interval near $t$ and a simple
look at the DOS plot provides a sort of sketch view of the spectrum of $A$.

Another, somewhat related, use of density  of states is in helping deploy
spectrum  slicing  strategies \cite{DDSpec,spectrumslicing,lsfeast}.  The  goal of  such  strategies  is  to
subdivide a given interval of  the spectrum into subintervals in order
to compute  the eigenvalues  in each subinterval  independently.  Note
that  this  is  often  done   to  balance  memory  usage  rather  than
computational load. Indeed, load balancing cannot be assured by just
having  slices with  roughly equal  numbers of  eigenvalues.  %In  fact spectrum slicing will help mainly in balancing memory usage.
With the availability of the spectral density function $\phi$, slicing the
 spectrum contained in an interval $[a, \ b]$ 
into $n_s$ subintervals can be  easily  accomplished. 
 Indeed, it suffices to find intervals $[t_i, \ t_{i+1}]$,
$i=0, \cdots, n_{s}-1$,  
 with $t_0 = a$ and $t_{n_s} = b$ such that 
\[ 
 \int_{t_i}^{t_{i+1}} \phi(t) dt = 
\frac{1}{n_{s}}\int_{a}^{b} \phi(t) dt , \quad i=0,1,\cdots, n_s-1 . 
\] 
See Fig. \ref{fig:dosslice} for an illustration and Section \ref{sec:slicing} for more details.

\begin{figure} 
\centerline{\includegraphics[width=0.62\textwidth]{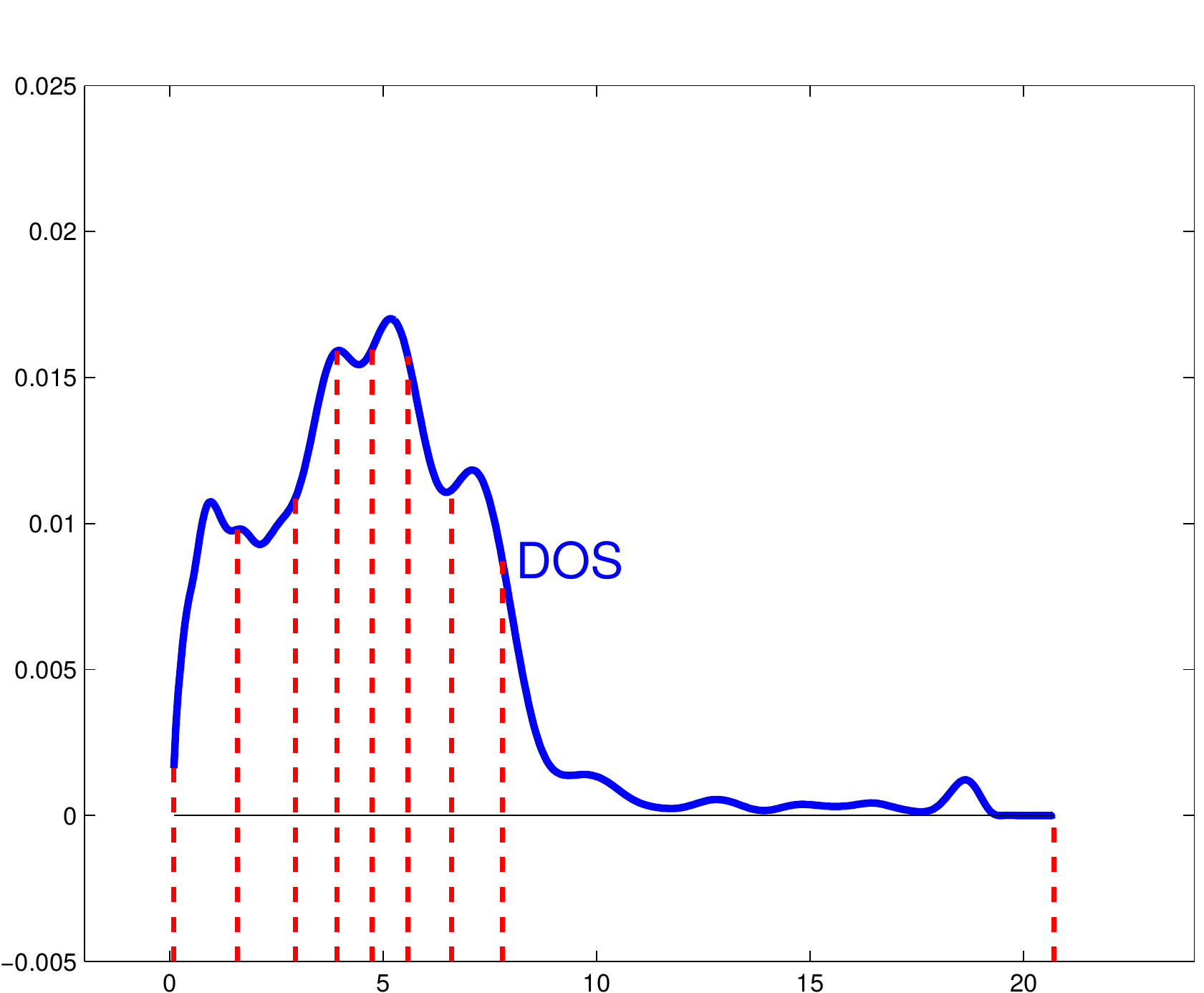}}
\caption{An illustration of slicing a spectrum into $8$ subintervals $[t_i,\ t_{i+1}]$ $(i=0,\ldots, 7)$.  The solid blue curve represents a smoothed density of states \text{(DOS)} and the dotted red lines separate the subintervals.}
\label{fig:dosslice} 
\end{figure}

A non-standard and important use of spectral densities is when estimating
numerical ranks of matrices
\cite{UbaruYSrank16,UbaruYSrankDL}. In many applications, a given 
$m \times n$ data matrix $A$ (say with $m\!>\!n$) 
is known to correspond to a phenomenon that should yield vectors lying in 
a low-dimensional space. With noise and approximations the resulting data 
is no longer of low-rank but it may be nearly low-rank in that its 
numerical rank is small. It may be important in these applications to 
obtain this numerical rank. In \cite{UbaruYSrank16,UbaruYSrankDL} the authors
developed a few heuristics that exploit the spectral density for this task.
The main idea is that for a nearly low-rank matrix,
the spectral density should be quite high near the origin of the matrix 
$A^T A$ and it should drop quickly before increasing again. The numerical rank
corresponds to the point when $\phi$ starts increasing again, i.e., when
the derivative of the DOS changes signs. This simple strategy provides
an efficient way to estimate the rank.

A straightforward way to obtain the spectral density of a given matrix
$A$ is to  compute all its eigenvalues but this  approach is expensive
for  large  matrices.  Effective   alternatives  based  on  stochastic
arguments   have  been   developed,  see,   \cite{DOSpaper16}  for   a
survey. Essentially  all the  methods described  in the  literature to
compute the DOS rely on performing  a number of products of the matrix
$A$  with random  vectors.  For  sparse matrices  or dense  structured
matrices   with  almost   linear  complexity   matrix-vector  products
\cite{smash,GREENGARD1997280}, these products are inexpensive and so a
fairly good  approximation of the  DOS can be  obtained at a  very low
cost. On the  other hand, not much  work has been done  to address the
same  problem  for  generalized  eigenvalue problems  \eq{eq:Pb}  A  x
=\lambda B x.   \en This paper focuses on this  specific issue as well
as on the related problem  on implementing spectrum slicing techniques
\cite   {spectrumslicing,Schofield2012497}.     From   a   theoretical
viewpoint the  problem may  appear to  be a  trivial extension  of the
standard   case.  However,   from   a   practical  viewpoint   several
difficulties  emerge, e.g.,  it is  now  necessary to  solve a  linear
system  with $B$  (or $A$)  each  time we  operate on  vectors in  the
stochastic  sampling   procedure  or  in  a   Lanczos  procedure.  For
large-scale problems  discretized from 3D models,  factorizing $B$ (or
$A$) tends to be prohibitively expensive and so this naturally leads to
the question: \emph{Is it  possible to completely avoid factorizations
  when computing the density of  states for (\ref{eq:Pb})?} As will be
seen the answer is yes, i.e., it is possible to get the DOS accurately
without any factorizations and at a  cost that is comparable with that
of standard problems in many applications. For example, the matrix $B$
is often the mass matrix in discretizations such as the Finite Element
Method (FEM).  An  important observation that is  often made regarding
these matrices is that they  are strongly diagonally dominant.

%%YS: the following comments has been removed.. I think we can compute
%%    eigenvalues of $A,B$ without factoring A or B.
%%We also
%% want  to  emphasize  that  this  issue is  in  contrast  with  solving
%% generalized   eigenvalue  problems,   where  avoiding   factorizations
%% altogether is, generally, not too practical.

%The approaches just mentioned can be trivially extended to calculate 
%the DOS for generalized eigenvalue problems of the form 
%\eq{eq:Pb}
% A x =\lambda B x,
%\en
%when $A$ and $B$ are both Hermitian (or real symmetric) and $B$ is further positive definite. 

%In important applications the matrix $B$ is often the mass matrix
%in  discretizations such as the Finite Element Method (FEM).
%An important observation that is often made regarding these matrices 
%is that they are strongly diagonally dominant. This leads to the question:

%% Notation:

In the remainder of the paper 
we will assume that $A$ and $B$ are Hermitian while, in addition,
 $B$ is positive definite. We will call 
$\lambda_j$, $j=1,2,\cdots,n$  the eigenvalues of the pencil $
(A,B)$, and assume that they are  labeled increasingly.  
We also denote by $u_{j}$ the eigenvector corresponding
to $\lambda_{j}$, so if $U=[u_1, u_2,\cdots, u_n]$ and
$\Lambda = \mbox{diag} (\lambda_1,\lambda_2,\cdots,\lambda_n)$, then 
the pencil $(A, B)$ admits the eigen-decomposition
\begin{align} 
  U^T A U  &= \Lambda 	\label{eqn:Aeigen1}  \\
  U^T B U  &=  I  \ .\label{eqn:Aeigen2}
\end{align}

The rest of  the paper is organized as  follows. Section \ref{sec:sym}
discusses  a  few  techniques  to  avoid  direct  factorizations  when
extending  standard   approaches  for   computing  the   DOS to the generalized
eigenvalue problem.  Section
\ref{sec:kpm}  presents   the  extension   of  the   classical  Kernel
Polynomial Method (KPM) and  Section \ref{sec:lan} studies the Lanczos
method  from  the  angle  of quadrature.  We  provide  some  numerical
examples  in  Section  \ref{sec:numerical} and  draw  some  concluding
remarks in Section \ref{sec:conclude}.

%We will propose two methods for calculating the DOS for generalized eigenvalue problems. The first one is based on the classical Kernel Polynomial Method (KPM), see, e.g.,
%~\cite{kpmsurvey2006,DOSpaper16},
%and the second one is based on the Lanczos method viewed from the angle of quadrature
%\cite{DOSpaper16}. 
%
%Before that, we first discuss a few techniques on how to avoid factorizations when extending standard approaches for computing the DOS.

\section{Symmetrizing the generalized eigenvalue problem}\label{sec:sym}
A common way to express the generalized eigenvalue problem
\nref{eq:Pb} is to multiply through by $B\inv$:
\eq{eq:stdB}
B\inv A x = \lambda x. 
\en 
This is now in the standard form but the matrix involved is  non-Hermitian.
However, as is well-known, the matrix $B\inv A$ is self-adjoint with
respect to the $B$-inner product and this observation allows  one to use
standard methods, such as the Lanczos algorithm,
that are designed for Hermitian matrices. 

Another way to extend standard approaches for computing the spectral density
is to transform the problem \nref{eq:Pb} into a standard one via the 
Cholesky factorization.
First, assume that  the Cholesky factorization of $B$ is available and
let it be written as $B=LL\trans$. 
Then the original  problem \nref{eq:Pb} can also be rewritten as
\eq{eq:stdL}
L\inv A L\invt y = \lambda y, \quad \mbox{with} \quad y=L\trans x,
\en 
which takes  the standard form with a Hermitian coefficient matrix.
This allows us to 
express the density of states from that of  a standard problem.
This straightforward solution faces a number of issues.
Foremost among these is the fact that the Cholesky factorization may not
be available or that it may be too expensive to compute. 
In the case  of FEM methods, the factorization of $B$ may be
too costly for $3D$ problems.

Note that the matrix square root factorization can also be used in the same way. Here 
 the original  problem \nref{eq:Pb} is transformed into the equivalent problem: 
\eq{eq:stdS}
B^{-1/2} A B^{-1/2}  y = \lambda y, \quad \mbox{with} \quad y=B^{1/2}  x,
\en 
which also assumes  the standard form with a Hermitian coefficient matrix.
 The square root factorization is usually expensive to compute and may appear
to be impractical at first. However, in the common situation mentioned above
where  $B$ is
strongly diagonally dominant, the action of $B^{-1/2}$ as well $B^{-1}$ on a vector can
be easily approximated by the matrix-vector product associated with a low degree polynomial in $B$. 
This is discussed next.

\subsection{Approximating actions of $B\inv$ and $B^{-1/2}$ on vectors}\label{sec:funapp}
As was seen above computing the DOS for a pair of matrices requires matrix-vector products with either 
 $B\inv A$, or $L\inv A L\invt$ 
or with $B^{-1/2} A B^{-1/2}$. 
Methods based on the first two cases can be implemented  with direct methods
but this   requires a factorization of $B$.
%% If direct methods are used 
%% in the first two cases, then $B$ must be factored (e.g., Cholesky) and the operation 
%% with $B\inv$ will amount to applying the forward and backward solves with the factors. 
Computing  the Cholesky,  or any  other  factorization of  $B$ is  not
always  economically  feasible for  large  problems.  It is  therefore
important to  explore alternatives  based on the  third case  in which
polynomial approximations of $B^{-1/2}$ are exploited.

All we need to  apply the methods described in this paper  is a way to
compute $B^{-1/2} v$ or $B\inv v$ for an arbitrary  vector $v$. These
calculations  amount  to evaluating  $f(B)v$  where  $f (\lambda  )  =
\lambda ^{-1/2} $  in one case and  $f(\lambda ) = 1/\lambda  $ in the
other.  Essentially the same method is  used in both cases, in that $f
(B)  v  $  is replaced  by  $f_k(B)v$  where  $f_k$  is an  order  $k$
polynomial   approximation  to   the  function   $f$  obtained   by  a
least-squares approach.  Computing $B^{-1/2} v$, is a problem that was
examined   at  length   in   the  literature   --   see  for   example
\cite{ChowEtAl,ChAnSa09,Higham-book} and references  therein.  Here we
use a  simple scheme that relies  on a Chebyshev approximation  of the
square root function in the interval $[a, \ b]$ where $a\!>\!0$.

Recall that any function that is analytic in $[a , \ b]$ can be expanded
in Chebyshev polynomials. To do so, the first step is to 
map $[a, \ b ] $ into the interval $[-1, \ 1]$, i.e., we impose the
change of variables from $\lambda \!\in\! [a, \ b] $ to $t \! \in \! [-1, \ 1]$:
\[
t = \frac{\lambda-c}{h} \quad \mbox{with} \quad 
c = \frac{a+b}{2}, \quad h = \frac{b-a}{2} \ . 
\]
In this way the function is transformed into a function $f$ with variables
in the interval $[-1,  \ 1]$. 
%%$f$ is then transformed into a function of variables in the 
%%interval $[-1, \ 1]$ which we denote by $g$. 
It is this $f$  that is approximated using the  truncated Chebyshev expansion:
\eq{eq:expg}
f_k(t) = \sum_{i=0}^k
\gamma_i T_i(t)
\quad \text{with} \quad 
\gamma_i =  \frac{2-\delta_{i0}}{\pi}  
\int_{-1}^1  \frac{f(s) T_i(s)}{\sqrt{1-s^2}} ds, 
\en
where $T_i(s)$ is the Chebyshev polynomial of the first kind of degree $i$.
Here   $\delta_{ij}$ is the Kronecker $\delta$ symbol so that 
 $2 - \delta_{k0} $ is equal to 1 when $k=0$ and to 2 otherwise.

Recall that $T_i$'s are orthogonal with respect to the
inner product 
\eq{eq:ChbProd} 
\left\langle p, q\right\rangle = \int_{-1}^1  \frac{p(s) q (s)}{\sqrt{1-s^2}} ds .
\en
We   denote by $\| . \|_\infty$ the supremum norm and by
 $\| . \|_C$ the $L_2$ norm associated with the above dot product:
\eq{eq:ChbNrm} 
\| p \|_C = \left[\int_{-1}^1  \frac{p(s)^2}{\sqrt{1-s^2}} ds \right]^{1/2}. 
\en
Note in passing that $T_i$'s do not have a unit norm with respect to 
\nref{eq:ChbNrm} but that the following  normalized sequence 
 is orthonormal:
\eq{eq:ChebNorm} 
\hat T_i(s) = \sqrt{\frac{2-\delta_{i0}}{\pi}}  T_i(s) ,
\en 
so that \nref{eq:expg} can be rewritten as $f_k(t) = \sum_{i=0}^k \hat \gamma_i \hat T_i(t)$ with
$\hat \gamma_i = \langle f(t), \hat T_i(t)\rangle$.

The integrals in \nref{eq:expg} 
 are computed using Gauss-Chebychev quadrature. 
The accuracy of the approximation and therefore the degree needed to
obtain a suitable approximation to use in replacement of $f(B)v$ depends
essentially on the degree of smoothness of $f$. One issue here is to determine the number of integration points
to use. Recall that when we use Gauss-Chebyshev quadrature with 
$\nu$ points,  the calculated integral is exact for all polynomials of degree $\le 2 \nu-1$.

%%YS added:

The reasoning for selecting $\nu$ is as follows. Let $p_{K} $ be the truncated Chebyshev expansion of $f$, with $K \gg k$. 
Then for $i\le k$ the coefficients $\hat \gamma_i$ for $i \le k$ are the same for
$p_k$ and for $p_{K}$ and they are:
\[
\hat \gamma_i = \left\langle f, \hat T_i \right\rangle 
= \left\langle f - p_{K}, \hat T_i \right\rangle  + \left\langle p_{K}, \hat T_i \right\rangle 
=  \left\langle p_{K}, \hat T_i \right\rangle .  
\]
The last equality is due to the orthogonality of the error to the $T_i$'s,
when $i\le K$.
Now observe that since $p_{K} (t) \hat T_i (t) $ is a polynomial of degree $\le K+k$ the integral 
$\langle p_{K}, \hat T_i \rangle$ will be computed exactly by the Gauss-Chebyshev rule 
as long as $K+k \le 2 \nu -1$, i.e., for $\nu \ge (K+k+1)/2$.
For example, when $K = 2k $ then for $\nu \ge (3k+1)/2$, $\hat \gamma_i$ \emph{will be the exact 
coefficient not for $f(t)$, but for $p_{2k}$ the degree $2k$ Chebyshev expansion which is usually much 
closer to $f$ than $p_k$.} 
While $\nu = \lceil (3k+1)/2 \rceil$ is usually sufficient, we prefer a lower margin for error and select
$\nu = 4k$ bearing in mind that the cost of quadrature is negligible.

\subsection{Analysis of the approximation accuracy}

Consider the two functions 
$ f_1(\lambda) = \lambda^{-1/2}$ and   $ f_2(\lambda) = \lambda^{-1}$ 
over  
$\lambda \! \in \! [a, \ b]$ where
$a\!>\!0$. It is assumed that the interval $[a, \ b]$ contains the spectrum of 
$B$ - with ideally $a = \lambda_{min} (B)$, $b = \lambda_{max} (B)$.
We set $c = (a+b)/2$, $h = (b-a)/2$. As mentioned above
we need to transform the interval $[a,\ b]$ into $[-1, \ 1]$, so the
transformed functions  being approximated  are in fact 
%%$g(t) = f(l(t))$, i.e., 
\begin{align} 
g(t) &= (c+h t)^{-1/2} , \label{eq:goft} \\
q(t) &= (c+h t)^{-1} , \label{eq:qoft} 
\end{align} 
with the variable $t$ now in $[-1, \ 1]$.
These two functions are clearly analytic in the interval $[-1,  \ 1]$
and they have  a singularity when $c+ht = 0$, i.e., at
$t_s = -c/h $ which is less than $-1$.
 Existing results in the literature will help analyze the convergence of the
truncated Chebyshev expansion in situations such as these, see, e.g.,
\cite{Trefethen-book}.

We can apply the result of Theorem~8.2 in the book \cite{Trefethen-book} to
show a strong convergence result. 
The Joukowsky transform
$(z+1/z)/2$ maps the circle $C(0,\rho)$ into an ellipse $E_\rho$,
with major semi-axis 
$(\rho +\rho \inv)/2$ and focii $-1, 1$. There are two values of $\rho$
that give the same ellipse and they are inverses of each other.
 We assume that $\rho\!>\!1$. The ellipse $E_\rho$
 is called the Bernstein ellipse in the framework of the theorem 
in \cite{Trefethen-book}  which is restated below for the present context. See Fig. \ref{fig:bernstein} for an illustration of Bernstein ellipses corresponding to different $\rho$'s.
 \begin{theorem}\label{thm:trf} 
\cite[Theorem 8.2]{Trefethen-book} 
 Let a function $f$ analytic in $[-1, \ 1]$ be analytically continuable
to the open Bernstein ellipse $E_\rho$ where it satisfies 
$|f(t)| \le M(\rho) $ for some $M(\rho) $. Then for each $k \ge  0$, its
truncated Chebyshev expansion $f_k$ (eq. \nref{eq:expg})  satisfies:
\eq{eq:trf} 
\| f - f_k \|_\infty \le \frac{2 M(\rho)  \rho^{-k}}{\rho -1}.
\en 
\end{theorem}

\begin{figure}[htb]
\centerline{
\includegraphics[width=0.7\textwidth]{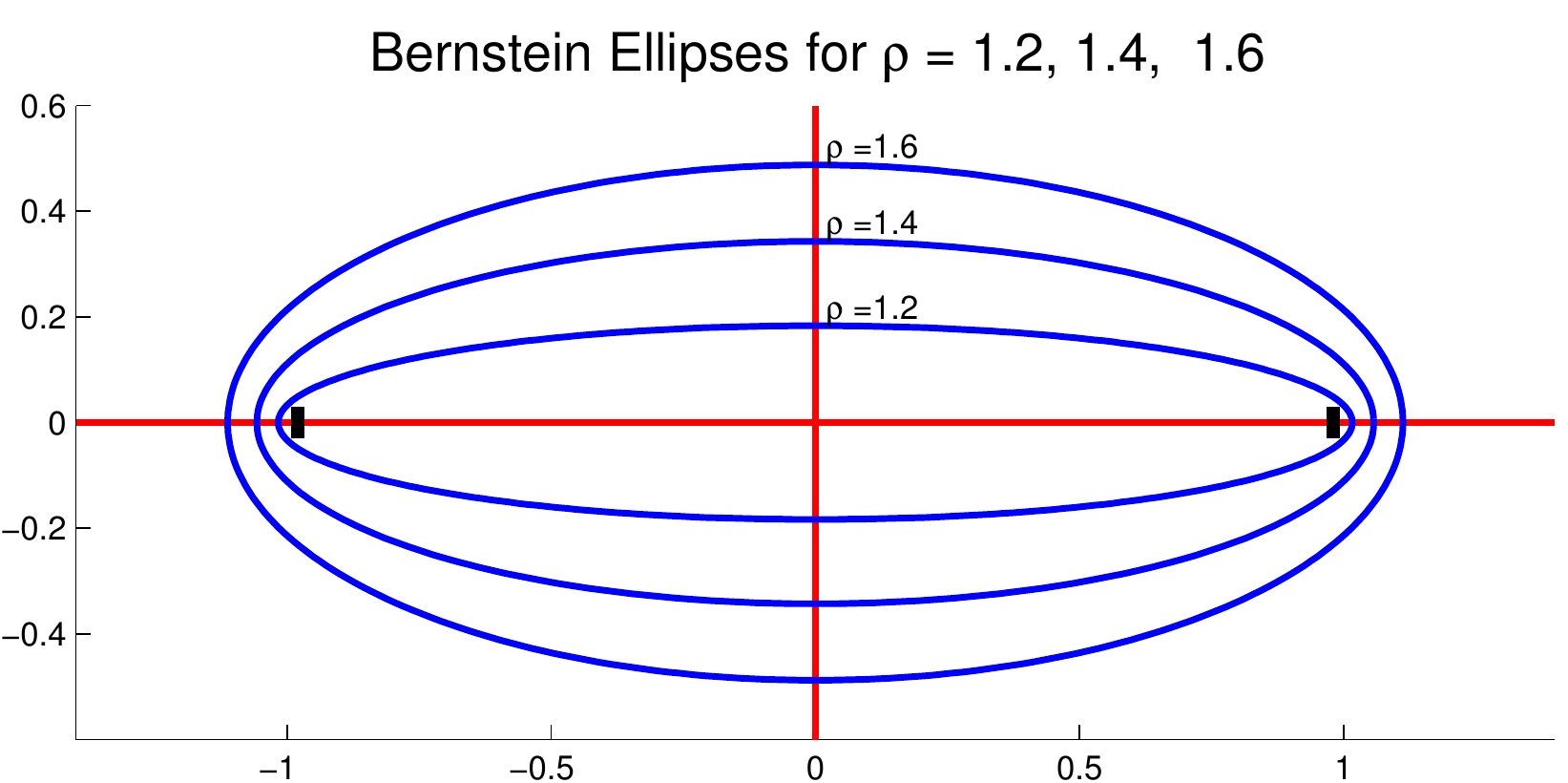}}
\caption{Bernstein ellipses for $\rho=1.2, 1.4, 1.6$.}
\label{fig:bernstein}
\end{figure}

The Bernstein ellipse should not contain the
point of singularity. Therefore, for the two functions under
consideration, we should take \emph{any} $\rho>1$ such that 
$(\rho + \rho\inv)/2 \!<\! c/h$, i.e., $\rho$ must satisfy: 
\eq{eq:rhoBnds} 
1 < \rho < \frac{c}{h} +  \sqrt{\left(\frac{c}{h}\right)^2 -1}.
\en
The next ingredient from the theorem is an upper bound $M(\rho)$ 
for $|f(t)|$ in $E_\rho$. In fact the maximum value of this
modulus is computable for both functions under consideration
and it is given in the next lemma.
\begin{lemma} \label{lem:MBnds}  
Let $\rho$ be given such that \nref{eq:rhoBnds} is satisfied. Then
the maximum modulii of the functions 
\nref{eq:goft} and \nref{eq:qoft} 
for $t \! \in \! E_\rho$ are given, respectively, by 
\begin{align} 
M_g (\rho) &= 
\frac{1}{\sqrt{c - h \frac{\rho + \rho\inv}{2}}} \label{eq:maxg}, \\
M_q (\rho) &= 
\frac{1}{c - h \frac{\rho + \rho\inv}{2}}  . \label{eq:maxq} 
\end{align} 
\end{lemma} 
\begin{proof}
Denote $d=c+ht$  the term inside the parentheses of \nref{eq:goft} and
\nref{eq:qoft} and 
write $t \! \in \! E_\rho$ as: 
$t = \half [\rho e^{i \theta} + \rho\inv e^{-i \theta}]$. Then
 $ d = c + h (\rho e^{i \theta} + \rho\inv e^{-i \theta})/2 $ and 
\begin{align*} 
|d|^2 &= (c + h t ) (c+ h \bar t) = c^2  +  h c (t+\bar t) + h^2 t \bar t \\
 &= c^2  +  h c (\rho + \rho\inv) \cos \theta 
+ \frac{h^2}{4} [\rho^2 + \rho^{-2} + 2 \cos (2 \theta)] . 
\end{align*} 
Observe that $\rho^2 + \rho^{-2} = (\rho + \rho\inv)^2 -2$ and $\cos (2 \theta) = 2 \cos^2 \theta -1 $. Therefore,
\begin{align*} 
|d|^2 
 &= c^2  +  h c (\rho + \rho\inv) \cos \theta 
+ \frac{h^2}{4} [(\rho + \rho^{-1})^2 - 2 (1-\cos (2 \theta))] \\ 
 &= c^2  +  h c (\rho + \rho\inv) \cos \theta 
+ \frac{h^2}{4} [(\rho + \rho^{-1})^2 - 4 (1-\cos^2 \theta)] \\ 
 &= \left[c  +  h  \frac{\rho + \rho\inv}{2} \cos \theta \right]^2 
+ \frac{h^2}{4} 
\left[ (\rho + \rho^{-1})^2 - 4\right] (1-\cos^2 \theta)  \\ 
 &= \left[c  +  h  \frac{\rho + \rho\inv}{2} \cos \theta \right]^2 
+ h^2 
\left[\left( \frac{\rho + \rho^{-1}}{2}\right)^2 - 1\right] \sin^2 \theta .
\end{align*}
Since $(\rho + \rho\inv)/2\!>\!1 $, the second term in brackets is positive
and it is then clear that the minimum value of $|d|^2$ is reached when
$\theta = \pi$ and the corresponding $|d|$ is 
$c - h (\rho + \rho\inv)/2$. Inverting this gives \nref{eq:maxq}.
Taking the inverse square root yields \nref{eq:maxg} and this completes the proof.
\end{proof} 

Note that, as expected, both maxima go to infinity as $\rho$ approaches
its  right (upper) bound given by \nref{eq:rhoBnds}.
We can now state the following theorem which simply applies Theorem 
\ref{thm:trf} to the functions \nref{eq:goft} and \nref{eq:qoft},
using the bounds for $M(\rho)$ obtained in Lemma \ref{lem:MBnds}. 
 \begin{theorem}\label{thm:Main}
 Let  $g$ and $q$  be the functions 
given by \nref{eq:goft} and \nref{eq:qoft} and let
$\rho$ be any real number that
satisfies the inequalities \nref{eq:rhoBnds}. 
Then the truncated Chebyshev expansions $g_{k_1}$  and $q_{k_2}$ of $g$ and $q$,
respectively, satisfy:
\begin{align} 
\| g - g_{k_1} \|_\infty  & \le \frac{2 \rho^{-k_1}}{(\rho -1) 
\sqrt{c - h \frac{\rho + \rho\inv}{2}}} \label{eq:trfg}, \\
\| q - q_{k_2} \|_\infty  & \le \frac{2 \rho^{-k_2}}
{(\rho -1) \left(c - h \frac{\rho + \rho\inv}{2}\right)}. \label{eq:trfq} 
\end{align} 
\end{theorem}  

Theorem~8.1 in \cite{Trefethen-book}, upon which Theorem~\ref{thm:trf} is
based, states that the coefficients $\gamma_i$ in \nref{eq:expg} 
decay geometrically, i.e.,
\eq{eq:trf4} 
|\gamma_k| \le 2 M(\rho) \rho^{-k}.
\en
Based on the above inequality, it is now possible to establish the following result for the approximation error of $g_{k_1}$ and $q_{k_2}$ measured in the Chebyshev norm.
\begin{proposition}\label{prop:infNrm}
Under the same assumptions as for Theorem~\ref{thm:Main}, 
the  
truncated Chebyshev expansions $g_{k_1}$  and $q_{k_2}$ of $g$ and $q$,
satisfy, respectively:
\begin{align} 
\| g - g_{k_1} \|_C & \le \sqrt{\frac{2 \pi}{\rho^2-1}}
\frac{ \rho^{-{k_1}}}{
\sqrt{c - h \frac{\rho + \rho\inv}{2}}} \label{eq:trfgI}, \\
\| q - q_{k_2} \|_C  & \le \sqrt{\frac{ 2\pi}{\rho^2-1}}
\frac{ \rho^{-k_2}}
{\left(c - h \frac{\rho + \rho\inv}{2}\right)}. \label{eq:trfqI} 
\end{align} 
\end{proposition}  
\begin{proof} 
For any function $f$ expandable as in \nref{eq:expg}, we have
\[
f(t) - f_k(t) = \sum_{i=k+1}^\infty \gamma_i T_i(t).
\] 
Because of the orthogonality of the Chebyshev polynomials and the inequality (\ref{eq:trf4}), we obtain
\begin{align*} 
\|f(t) - f_k(t) \|_C^2 
&= \sum_{i=k+1}^\infty |\gamma_i|^2  \| T_i(t)\|_C^2 \le \sum_{i=k+1}^\infty 4 M(\rho)^2 \rho^{-2 i} \frac{\pi}{2} \\
&\le 2 M(\rho)^2 \pi \rho^{-2 (k+1)} \frac{1}{1-\rho^{-2}}= 2 M(\rho)^2 \pi \rho^{-2 k } \frac{1}{\rho^2-1}. 
\end{align*} 
Taking the square root and replacing the values of $M(\rho)$ 
from Lemma~\ref{lem:MBnds} yield the two inequalities. 
\end{proof} 

Both Theorem \ref{thm:Main} and Proposition \ref{prop:infNrm} show that the Chebyshev expansions $g_{k_1}$ and  $q_{k_2}$ converge geometrically. The plot in Fig.~\ref{fig:Deg8LS_sqrt} indicates that a low degree is sufficient to reach a reasonable accuracy for the needs of computing the DOS.

\begin{figure} 
\includegraphics[width=0.49\textwidth]{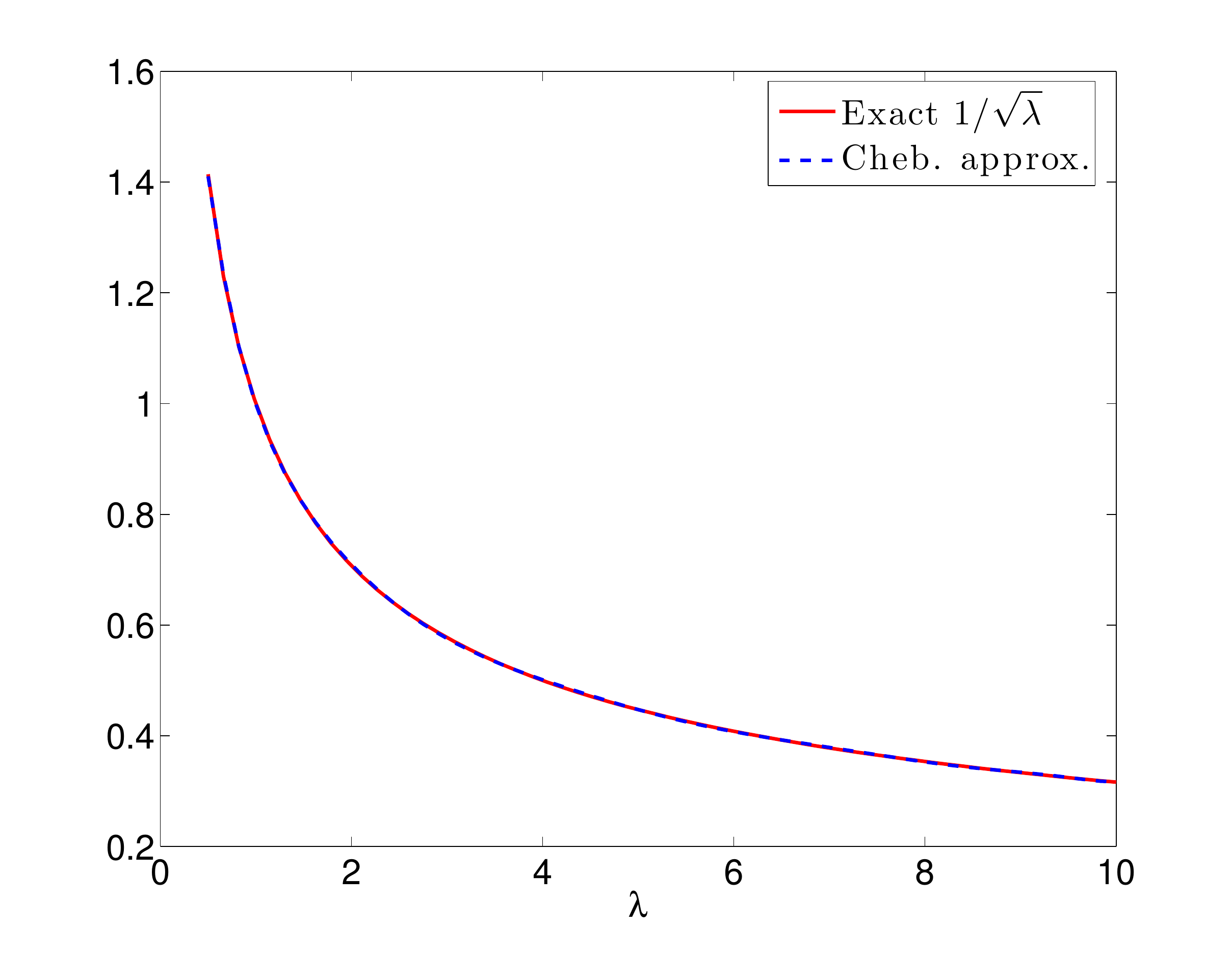}
\includegraphics[width=0.49\textwidth]{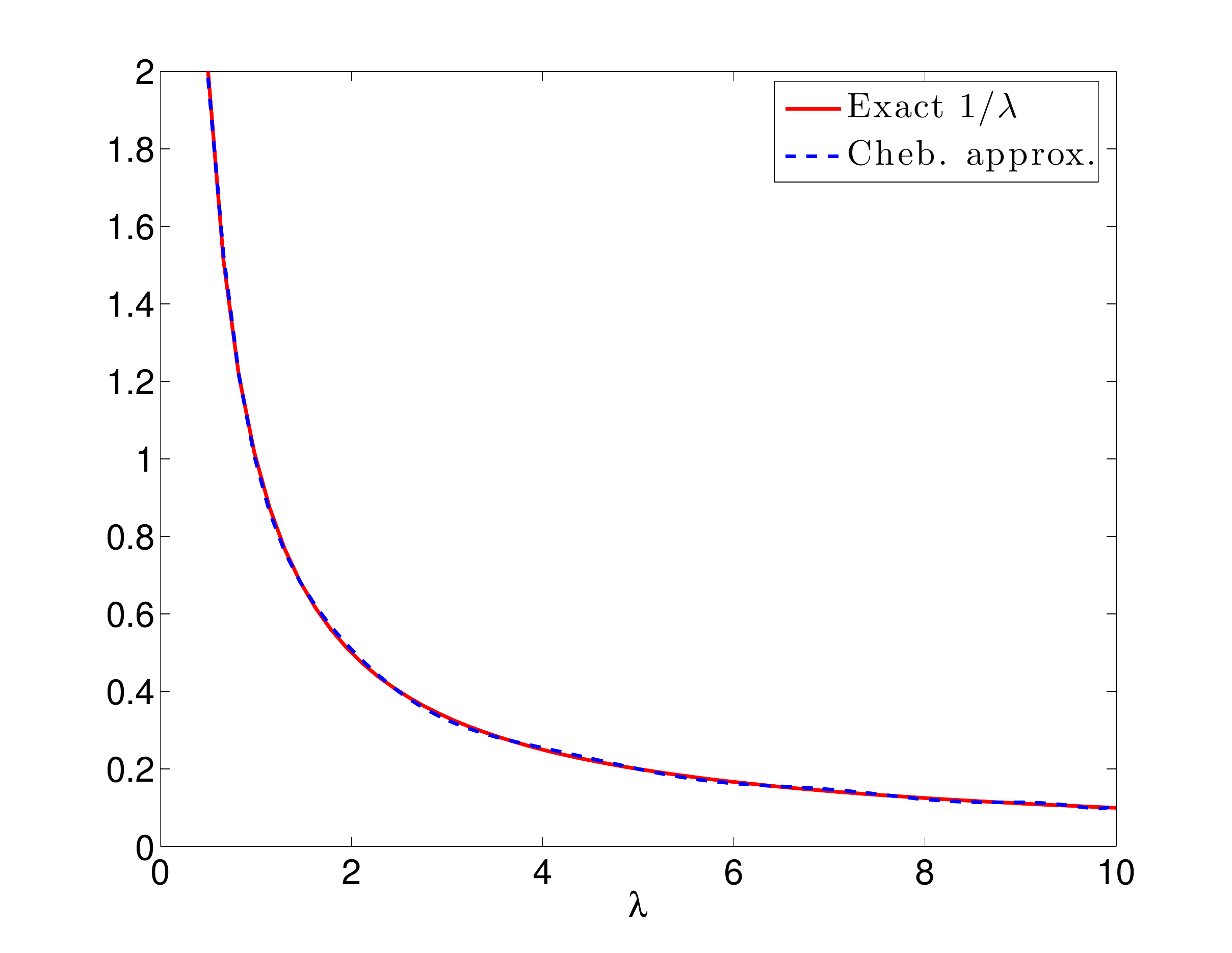}
\label{fig:Deg8LS_sqrt}
\caption{Degree $10$ Chebyshev polynomial approximations to $1/\sqrt{\lambda}$ and $1/\lambda$ on the interval $[0.5,\ 10]$.}
\end{figure}

\subsection{Bounds involving the condition number of $B$}
Theorem \ref{thm:Main} shows that the \emph{asymptotic} convergence rate increases with
$\rho$. However, the ``optimal" value of $\rho$, i.e., the one that yields the smallest bounds in
\nref{eq:trfg} or \nref{eq:trfq},
 depends on $k_i$ and is hard to choose in practice.  Here, we will discuss two  simple 
choices for $\rho$ that will help analyze the convergence.
First,  we select $\rho = \rho_0 \equiv c/h$  which
satisfies the bounds \nref{eq:rhoBnds}. It leads to 
\eq{eq:Mrho0}
M_g (\rho_0 ) = 
\frac{\sqrt{2}}{\sqrt{c - h^2/c}} ,
\qquad 
M_q (\rho_0) = (M_g (\rho_0 ))^2 \ .
\en 
Note that in the context of our problem, if we denote 
by $\lambda_{max} (B), \lambda_{min} (B) $ the largest and smallest
eigenvalues of $B$ and by $\kappa $ its spectral condition number, then
\[ \rho_0 = c/h 
= \frac{\lambda_{max} (B) + \lambda_{min} (B) }
{\lambda_{max} (B) - \lambda_{min} (B) } 
= \frac{\kappa +1} {\kappa - 1} ,
\]
and therefore, for this choice of $\rho$,
 the bounds of the theorem evolve asymptotically like 
$(\frac{\kappa- 1} {\kappa +1})^{k} $.
A slightly more elaborate selection of $\rho$  is the value for which
$(\rho+\rho\inv)/2  = \sqrt{c/h}$ which is
$\rho_1 = \sqrt{c/h} + \sqrt{(c/h) - 1} $.
 For $t\ge 1$,  \ 
 $t+\sqrt{t^2-1}$ is an increasing function and therefore,
$1 \le \rho_1 \le (c/h) + \sqrt{(c/h)^2 - 1}$ and so the bounds 
\nref{eq:rhoBnds} are satisfied. With this we get:
\[
M_g (\rho_1 ) = 
\frac{1}{\sqrt{c - \sqrt{h c}}} , 
\qquad 
M_q (\rho_1) = (M_g (\rho_1 ))^2\ .
\]
In addition, we note that  $\rho_1 $ can also be expressed in terms of
the  spectral   condition  number   $\kappa$  of   $B$ as follows:  $   \rho_1  =
[\sqrt{\kappa  +1 }  + \sqrt{2}]/[\sqrt{\kappa  -1 }]$.  The resulting
term  $\rho_1^{-k}$ in  \nref{eq:trfg} and  \nref{eq:trfq} will  decay
much faster than $\rho_0^{-k}$ when  $\kappa$ is larger than $2$. Both
choices of $\rho$  show that for a fixed degree  $k$, a smaller $\kappa$
will result in  faster  convergence.

If $B$ is  a mass matrix obtained from a  FEM discretization, $\kappa$
can  become very  large for  a  general nonuniform  mesh.  One  simple
technique to reduce  the value of $\kappa$ is to  use diagonal scaling
\cite{mass,Wathen1987,Wathen2008}.             Suppose             $D=
\operatorname{diag}(B)$,  then by congruence,
 the  following problem  has  the  same
eigenvalues as (\ref{eq:Pb}) 
\begin{equation}
D^{-1/2}AD^{-1/2} z = \lambda D^{-1/2}BD^{-1/2}z, \quad \mbox{with} \quad z=D^{1/2}  x.
\label{eq:scaling}
\end{equation}
It was shown in  \cite{Wathen1987,Wathen2008} that, for any conforming
mesh   of    tetrahedral   (P1)   elements   in    three   dimensions,
$\kappa(D^{-1/2}BD^{-1/2})$  is  bounded by  $5$  and  for a  mesh  of
rectangular    bi-linear   (Q1)    elements    in   two    dimensions,
$\kappa(D^{-1/2}BD^{-1/2})$ is bounded by $9$. Moreover, this diagonal
scaling  technique  has also been exploited   to  reduce  the  spectral
condition  number   of  graph  Laplacians  in   the  network  analysis
\cite{graph16}.  As a  result, we  will always  preprocess the  matrix
pencil $(A,B)$ by diagonal scaling before computing the DOS.

With the approximations in (\ref{eq:expg}), we obtain 
\begin{align}
B^{-1} &\approx g_{k_1}(B) :=\sum_{i=0}^{k_1}
\gamma_i T_i[(B-cI)/h],\label{eq:eq1} \\
B^{-1/2} &\approx q_{k_2}(B):=\sum_{i=0}^{k_2}
\beta_i T_i[(B-cI)/h].\label{eq:eq2}
\end{align}

Using  the above  approximations to replace $B^{-1}$ and $B^{-1/2}$
in (\ref{eq:stdB})  and (\ref{eq:stdS}), will amount to  computing
the DOS of the modified problem
\begin{equation}
\label{eq:g_m}
g_{k_1}(B)A \tilde{x}= \tilde{\lambda} \tilde{x}.
\end{equation}
Therefore,  it  is  important  to   show  that  the  distance  between
$\tilde{\lambda}$ and $\lambda$  is small when  $g_{k_1}$ and $q_{k_2}$ 
reach a certain accuracy. We will need the following 
perturbation result for   Hermitian definite pencils.

 \begin{theorem}\label{thm:weyl} 
   \cite[Theorem  2.2]{NAKATSUKASA2010242}  Suppose that  a  Hermitian
   definite  pencil $(A,B)$  has eigenvalues  $\lambda_1\leq \lambda_2
   \leq \dots \leq  \lambda_n$. If $\Delta A, \Delta  B$ are Hermitian
   and $||\Delta  B||_2<\lambda_{min}(B)$, then  $(A+\Delta A,B+\Delta
   B)$   is   a   Hermitian    definite   pencil   whose   eigenvalues
   $\hat{\lambda}_1\leq     \hat{\lambda}_2     \leq    \dots     \leq
   \hat{\lambda}_n$ satisfy
\begin{align}\label{eq:weylbound} 
| \lambda_i - \hat{\lambda}_i| \leq \frac{||\Delta A||_2}{\lambda_{min}(B)}+\frac{|\lambda_i|\lambda_{min}(B)+||\Delta A||_2}{\lambda_{min}(B)(\lambda_{min}(B)-||\Delta B||_2)}||\Delta B||_2.
\end{align}
\end{theorem}  
In  the context  of  (\ref{eq:g_m}), the  perturbation  $\Delta B$  in
Theorem \ref{thm:perturb}  corresponds to the  approximation error
of  $g_{k_1}(B)$  to  $B^{-1}$.  This  implies  that  we  can  rewrite
(\ref{eq:g_m}) in the form of
\[
A \tilde{x}= \tilde{\lambda} (B+\Delta B)\tilde{x} \quad \text{with} \quad \Delta B = (g_{k_1}(B))^{-1}-B,
\]
and then apply Theorem \ref{thm:weyl} to prove the following perturbation bound for (\ref{eq:g_m}).
\begin{theorem}\label{thm:perturb} 
Let $\lambda_1\leq \lambda_2 \leq \dots \leq \lambda_n$ be the eigenvalues of $B^{-1}A$ and $\hat{\lambda}_1\leq \hat{\lambda}_2 \leq \dots \leq \hat{\lambda}_n$ be the eigenvalues of $g_{k_1}(B)A$.
If $||g-g_{k_1}||_{\infty} \leq \tau$ and $||B||_2 \leq 1/\tau$, then we have
 \begin{align}\label{eq:weylbound2} 
| \lambda_i - \hat{\lambda}_i| \leq \frac{|\lambda_i|}{\lambda_{min}(B)-||\Delta B||_2}||\Delta B||_2,
\end{align}
with $||\Delta B||_2\leq \frac{||B||^2_{2}}{1-||B||_2\tau}\tau$.
\end{theorem}
\begin{proof}
Denote by $\theta_{1},\theta_{2},\ldots,\theta_{n}$ the eigenvalues of $B^{-1}$ and $\hat{\theta}_{1},\hat{\theta}_{2},\ldots,\hat{\theta}_{n}$ the eigenvalues of $g_{k_1}(B)$. Since $||g-g_{k_1}||_{\infty} \leq \tau$, we have
\begin{equation}
||B^{-1}-g_{k_1}(B)||_2 = \max_{i} |\theta_i-\hat{\theta}_i|\leq \tau.
\end{equation}
On the other hand, we know that
\begin{align}\label{eq:deltaB}
||\Delta B||_2=||B-(g_{k_1}(B))^{-1}||_2 &= \max_{i} |1/\theta_i-1/\hat{\theta}_i|  \leq \frac{\tau}{\theta_{1}(\theta_{1}-\tau)}=\frac{||B||^2_{2}}{1-||B||_2\tau}\tau. 
\end{align}
Substituting $||\Delta A||_2$ and $||\Delta B||_2$ in (\ref{eq:weylbound}) with 0 and (\ref{eq:deltaB}), respectively, we obtain the bound (\ref{eq:weylbound2}).
\end{proof}

Theorem  \ref{thm:perturb} indicates  that  if the  degree of  the
Chebyshev  expansions  is  chosen  in  such  a  way  that  the  bounds
(\ref{eq:maxg}--\ref{eq:maxq}) are  less than or equal  to $\tau$, the
eigenvalues  of  (\ref{eq:g_m}) would  be  close  enough to  those  of
(\ref{eq:stdB}). In the next two sections,  we will show how to extend
the  standard   algorithms  for  computing  the   DOS  to  generalized
eigenvalue problems of the form (\ref{eq:g_m}).

\section{The Kernel Polynomial Method} \label{sec:kpm} 
The Kernel Polynomial Method (KPM) is an effective technique 
proposed by physicists and chemists in the mid-1990s 
~\cite{DraboldSankey1993,ParkerZhuHuangEtAl1996,SilverRoder1994,SilverRoder1997,SilverRoederVoterEtAl1996,
Wang1994} to calculate the DOS of a Hermitian matrix $A$.
Its essence is to expand the function $\phi$ in 
\nref{eq:DOS0}, which is a sum
of Dirac $\delta$-functions, into   Chebyshev polynomials.

\subsection{Background: The KPM for standard eigenvalue problems} 
\label{sec:kpmstd} 
As is the case for all methods which rely on Chebyshev expansions,
a change of variables is first performed to map the interval
$[ \lambda_{\min},\ \lambda_{\max}]$ into $[-1, \ 1]$.
We assume this is already performed and so
 the eigenvalues are in the interval $[-1, \ 1]$.  
 To estimate the spectral density function \nref{eq:DOS0}, the KPM method approximates $\phi(t)$ by a finite expansion 
in a basis of orthogonal polynomials, in this case, 
 Chebyshev polynomials of the first kind.  Following
 the Silver-R\"oder paper~\cite{SilverRoder1994},  we include, 
for convenience,  
the inverse of the weight function into the spectral density function, 
so we expand instead the distribution:
\eq{eq:tdos}
\hat \phi(t)  = 
\sqrt{1-t^2} \phi (t) =
\sqrt{1-t^2}  \times \frac{1}{n}
\sum_{j=1}^n  \delta(t - \lambda_j).
\en 
Then, we have the (full) expansion
\eq{eq:exp1} 
\hat \phi(t)  = \sum_{k=0}^\infty  \mu_k T_k(t), 
\en 
where the expansion coefficients $\mu_k$ are formally defined by
%\begin{eqnarray}
\begin{align}  
\mu_k & = 
 \frac{2-\delta_{k0}}{\pi} \int_{-1}^1 \frac{1}{\sqrt{1 - t^2}} T_k(t) 
\hat \phi (t) dt 
\nonumber  = \frac{2-\delta_{k0}}{n\pi}  \sum_{j=1}^n  T_k(\lambda_j) . 
\label{eq:expmu} 
\end{align} 
%%\end{eqnarray}

Thus, apart from the scaling factor $(2-\delta_{k0})/(n\pi)$, 
 $\mu_k$ is  the trace of $T_k(A)$ and this can be estimated
by various methods including, but not limited to, stochastic approaches.
There are  variations on this idea starting with the use
of different orthogonal polynomials, to alternative ways in which the 
traces can be estimated.

%%Let us define the expansion
%%\eq{eq:Tcheb}
%%f_m (t) = \sum_{j=0}^\infty \mu_j T_j(t) .  
%%\en
The standard   stochastic  argument   for
estimating $\tr (T_k(A))$,
see~\cite{Hutchinson1989,SilverRoder1994,TangSaad2012},   
entails generating  a large number of random
vectors $v^{(1)}_{0}, v^{(2)}_{0},  \cdots, v^{(\nvec)}_{0}$ with each
component obtained from a normal  distribution with zero mean and unit
standard  deviation,   and  each   vector  is  normalized   such  that
$\norm{v^{(l)}_{0}}_2=1,l=1,\ldots,\nvec$.  The  subscript $0$  is added
to indicate that the vector has not been multiplied by the matrix $A$.
Then   we  can   estimate  the   trace  of   $T_k  (A)$   as  follows:
\eq{eq:guesstrace}      \tr     (      T_k      (A)     )      \approx
\frac{1}{\nvec}\sum_{l=1}^{\nvec}  \left(v^{(l)}_{0}\right)^T T_k  (A)
v^{(l)}_{0},   \en  
where the error decays as $\frac{1}{\sqrt{n_{nev}}}$ \cite{Hutchinson1989}.
Then  this  will lead  to  the  desired  estimate:
\eq{eq:muEst}     \mu_k    \approx     \frac{2-\delta_{k0}}{n\pi\nvec}
\sum_{l=1}^{\nvec}  \left(v^{(l)}_{0}\right)^T  T_k  (A)  v^{(l)}_{0}.
\en
 
 Consider the computation of each term 
$v_{0}^T T_k (A) v_{0}$ (the superscript $l$ is dropped for
simplicity).  The 3-term recurrence of the
Chebyshev polynomial:
$T_{k+1} (t) = 2 t T_{k} (t) - T_{k-1} (t)$ 
 can be exploited to compute $T_k(A)v_{0}$, 
%%
%%\eq{eq:3term}
%%T_{k+1} (A) v_{0} = 2 A T_k(A) v_{0} - T_{k-1}(A) v_{0}.
%\en
so that, if we let $v_{k} \equiv T_k(A) v_{0}$, we have
\eq{eq:3term2}
 v_{k+1} = 2 A v_{k} - v_{k-1}.
 \en
%%Once the scalars $\{\mu_k\}$ are determined, we would in theory get the 
%%expansion for 
%%$ \phi(t) = \frac{1}{\sqrt{1 - t^2}} \hat{\phi} (t) $. Practically however, 

The approximate density of states will be limited to Chebyshev polynomials of
degree $m$, so $\phi$ is approximated by the truncated expansion:
\eq{eq:tdosFinal} 
\wt{\phi}_m (t) = 
\frac{1}{\sqrt{1-t^2}} \sum_{k=0}^m \mu_k T_k (t) .
\en 
It has been proved in \cite{Lin2017} that the expansion error in (\ref{eq:tdosFinal}) decays as $\rho^{-m}$ for some constant $\rho\!>\!1$.

For a general matrix $A$ whose eigenvalues are not necessarily in the interval
$[-1, \ 1]$, a linear transformation is first applied to $A$ to bring
its eigenvalues to the desired interval. Specifically, we will apply the 
method to the matrix
\begin{equation}
\tilde{A} = \frac{A - c I}{h},
\label{eq:bounds}
\end{equation}
where
\begin{equation}
c = \frac{\lambda_{\min} + \lambda_{\max}}{2} \ , 
\quad 
h = \frac{\lambda_{\max} - \lambda_{\min}}{2}.
	\label{eqn:ch}
\end{equation}
It is important to ensure that the eigenvalues of $\tilde{A}$ are within the
interval $[-1, \ 1]$. In an application requiring a similar
approach~\cite{ZhouSaadTiagoEtAl2006}, we obtain the upper and lower
bounds of the spectrum from Ritz values provided by a standard Lanczos
iteration.  We ran $m$ Lanczos steps but extended the interval
$[\lambda_{\min}, \ \lambda_{\max}]$ by using the bounds obtained from the
Lanczos algorithm. Specifically,  the upper bound is set to $\wt{\lambda}_{m} + \eta$ where 
$\eta = \| (A - \wt{\lambda}_{m} I) \wt{u}_m \|_2$,
and $(\wt{\lambda}_m, \wt{u}_m)$ is the (algebraically) 
largest Ritz pair of $A$. In a similar way, the lower bound is set to $\wt{\lambda}_{1} - \beta$ where $\beta = \| (A - \wt{\lambda}_{1} I) \wt{u}_1 \|_2$ and $(\wt{\lambda}_1, \wt{u}_1)$ is the (algebraically) 
smallest Ritz pair of $A$. To summarize, we outline the major steps of the KPM for approximating 
the spectral density of a  Hermitian matrix in Algorithm~\ref{alg:kpm}. 
%Specifically,  the upper bound for $\lambda_{\max}$, for example,
%is set to $[\wt{\lambda}_{n}, \wt{\lambda}_{n} + \eta]$ where 
%%-YS error: norm instead of | . | 
%$\eta = \| (A - \wt{\lambda}_{n} I) \wt{u}_n \|$,
%and $(\wt{\lambda}_n, \wt{u}_n)$ is the (algebraically) 
%largest Ritz pair of $A$.
%This is plausible because we know that there is an eigenvalue in the interval 
%$[\wt{\lambda}_{n}, \wt{\lambda}_{n} + \eta]$ (see e.g., \cite{Saad2011}.) Even though there is no 
%theoretical guarantee that $[\wt{\lambda}_{n}, \wt{\lambda}_{n} + \eta]$ is indeed an upper bound 
%for $\lambda_{\max}$, in practice it often is.  Refined bounds which invoke 
%the square of $\eta$ (\cite{Saad2011}) can also be used. 
%Note that $\eta$ can be obtained without
%explicitly computing the eigenvector $\wt{u}_n$.

\begin{algorithm}
\begin{small}
\begin{center}
  \begin{minipage}{5in}
%%\begin{tabular}{p{0.5in}p{4.3in}}
\begin{tabular}{ll}
{\bf Input}:  &  \begin{minipage}[t]{4.0in}
                 A Hermitian matrix $A$, a set of points $\{t_{i}\}$
                 at which DOS is to be evaluated, the degree $m$ of
                 the expansion polynomial
                 \end{minipage}\\
{\bf Output}: &  Approximate DOS evaluated at $\{t_i\}$ \\
\end{tabular}
\begin{algorithmic}[1]
\STATE Compute the upper bound and the lower bound of the spectrum of $A$
\STATE Compute $c$ and $h$ in (\ref{eqn:ch}) with those bounds
\STATE Set $\mu_k = 0 $ for $k=0,\cdots, m$
\FOR {$l=1:\nvec$}
\STATE Select a new random vector $v_0^{(l)}$
\FOR {$k=0:m$}
\STATE Compute $T_{k} ((A-cI)/h) v_0^{(l)}$ using 3-term recurrence \nref{eq:3term2}
\STATE Update $\mu_k$ using \nref{eq:muEst}
\ENDFOR
\ENDFOR
\STATE Evaluate the average value of $\{\wt{\phi}_{m}((t_i-c)/h)\}$ at the given set
of points $\{t_{i}\}$ using~\eqref{eq:tdosFinal}
\end{algorithmic}
\end{minipage}
\end{center}
\end{small}
\caption{The Kernel Polynomial Method}
\label{alg:kpm}
\end{algorithm}

\subsection{The KPM for generalized eigenvalue problems}
\label{sec:kpmgen} 
We now  return to the generalized  problem \nref{eq:Pb}.  Generalizing
the KPM algorithm to this case is straightforward when the square root
factorization  $B=S^2$ or  the  Cholesky  factorization $B=LL^{T}$  is
available:  we  just  need  to use  Algorithm~\ref{alg:kpm}  with  $A$
replaced by  $S\inv A  S\inv$ or $L^{-1}AL^{-T}$.  In this  section we
only  discuss   the  case where a square  root   factorization is used.  
The alternative of  using the  Cholesky  factorization can  be carried out in  a
similar  way.  Clearly  $S\inv  A   S\inv$  needs  not  be  explicitly
computed. Instead, the product $S\inv A S\inv w$ that is required  when computing
$T_k((S\inv A S\inv - cI)/h) v_0^{(l)}$  in Line 7 of Algorithm 1, can
be  approximated  by  matrix-vector   products  with  $q_{k_2}(B)$  in
(\ref{eq:eq2}) and the matrix $A$.

%What if the factorization is not available or is too expensive to compute? 
%All that we need to make the KPM work is to estimate the scaled trace 
%\nref{eq:expmu} for $k=0, \cdots, m$. For we need to estimate the trace
%of $T_k(L\inv A L\invt) $ which is equal to the trace of $T_k(B\inv A)$.

%All that we need to make the KPM work is to estimate the trace of $T_k((S\inv A S\inv - cI)/h)$ for $k=0, \cdots, m$. %For we need to estimate the trace
%of $T_k(L\inv A L\invt) $ which is equal to the %trace of $T_k(B\inv A)$
%The important point here is that although it
%suggests that we take special vectors vectors of the form $z = L\invt v$,
%where $v$ is a random vector, see Relation~\ref{eq:rel1}, this is not 
%necessary because the trace estimator will work for $T_k(B\inv A)$.
The important point here is that if we simply follow the 3-term recurrence (\ref{eq:3term2}) and let $v_{k} \equiv T_k((S^{-1}AS^{-1}-cI)/h) v_{0}$, we have
\eq{eq:3term3} v_{k+1} = 2 \frac{S^{-1}AS^{-1}-cI}{h} v_{k} - v_{k-1}. \en 
This implies that the computation of each $v_k$ will involve two matrix-vector products with $S^{-1}$ and one matrix-vector product with $A$. On the other hand, premultiplying both sides of (\ref{eq:3term3}) with $S^{-1}$ leads to 
\begin{align*}
 S^{-1}v_{k+1} &= 2 S^{-1}\frac{S^{-1}AS^{-1}-cI}{h} v_{k} - S^{-1}v_{k-1}\\
  & = 2 \frac{B^{-1}A-cI}{h} S^{-1}v_{k} - S^{-1}v_{k-1}.
\end{align*} 
Denoting by $w_k := S^{-1}v_{k}$,  we obtain another 3-term recurrence
\eq{eq:3term4}  w_{k+1}   =  2\frac{B^{-1}A-cI}{h}w_{k}-w_{k-1}  \quad
\text{with}\quad w_0 =  S^{-1}v_{0}.  \en Now the  computation of each
$w_{k}$ only involves one matrix-vector  product with $B^{-1}$ and one
matrix-vector  product with  $A$. Since  Theorem \ref{thm:Main}  shows
that both the approximation errors of $g_{k_1}$ and $q_{k_2}$ decay as
$\rho^{-k_i}$,  this indicates  that the  same approximation  accuracy
will  likely  lead  to  roughly  the same  degree  for  $g_{k_1}$  and
$q_{k_2}$.    As   a    result,   recurrence    (\ref{eq:3term4})   is
computationally more economical than recurrence (\ref{eq:3term3}) when
we replace  $B^{-1}$ and $S^{-1}$ with  $g_{k_1}(B)$ in (\ref{eq:eq1})
and  $q_{k_2}(B)$   in  (\ref{eq:eq2}),  respectively.  In   the  end,
$v_{0}^{T}T_{k}((S^{-1}AS^{-1}-cI)/h)v_{0}$  in   (\ref{eq:muEst})  is
computed as $w_{0}^{T}Bw_{k}$.

Similarly, if Cholesky factorization of $B$ is applied, then the following 3-term recurrence is preferred in actual computations
\eq{eq:3term5}
w_{k+1} = 2\frac{B^{-1}A-cI}{h}w_{k}-w_{k-1} \quad \text{with}\quad w_0 = L^{-T}v_{0}.
\en
 
%In the end the KPM algorithm for the DOS of generalized eigenvalue
%problems with a $B$ matrix that is invertible, is just Algorithm
%\ref{alg:kpm} in which the matrix $T_k(A)$ in Line~5 is replaced by
%$T_k(B\inv A)$.
% Next section addresses the problem of computing
%$B\inv v$ effectively.

\section{The Lanczos method for Density of States}\label{sec:lan}
The well-known connection between the Gaussian Quadrature and the Lanczos
algorithm has also been exploited to compute the DOS
\cite{DOSpaper16}. We first review the method for standard problems
before extending it to matrix pencils.

%%-----------------------------------------------------------------------

\subsection{Background: The Lanczos procedure for the standard DOS}
The Lanczos algorithm  builds an orthonormal basis 
$V_m = [v_1, v_2, \cdots, v_m]$ for the \emph{Krylov subspace:} 
$ \Span \{v_1, Av_1, \cdots, A^{m-1} v_1 \} $ with an initial vector $v_1$. See Algorithm~\ref{alg:LanA} for a summary. 

\begin{algorithm}[H]
\begin{algorithmic}[1]  \label{alg:LanA}
\caption{Lanczos algorithm for a Hermitian matrix $A$}
\STATE Choose an initial vector $v_1$ with $\norm{v_1}_2=1$ and
set $\beta_1=0$, $v_0=0$
\FOR {$j=1,2,\ldots, m$}
\STATE $w:=Av_j-\beta_j v_{j-1}$
\STATE $\alpha_j=(w,v_j)$
\STATE $w:=w-\alpha_j v_j$
\STATE Full reorthogonalization: $w:=w-\sum_i(w,v_i)v_i$ for $i\le j$
\STATE $\beta_{j+1}= \| w\|_2$ 
\STATE If $\beta_{j+1}==0$ restart or exit 
\STATE $v_{j+1}:= w / \beta_{j+1}$
\ENDFOR
\end{algorithmic}
\end{algorithm}

At the completion of  $m$ steps of Algorithm \ref{alg:LanA}, we end up with the 
factorization  $V_m^T A V_m = T_m$  - with
\[ 
T_m = \begin{pmatrix} \alpha_1 & \beta_2   &        &       &     &    \cr
                \beta_2  & \alpha_2  &\beta_3 &       &     &    \cr
                        & \beta_3  & \alpha_3&\beta_4 &     &    \cr
                        &           &    .    &  .    &  .   &    \cr
                        &           &         &  .    &  .   &\beta_m   \cr
                        &           &         &       & \beta_m&\alpha_m  
\end{pmatrix} .
\] 
Note that the vectors $v_j$, for $j=1,\cdots,m, $ satisfy the 3-term recurrence
\[ 
\beta_{j+1} v_{j+1}  = A v_j - \alpha_j v_j - \beta_j v_{j-1}.
\]
In theory the $v_j$'s defined by this recurrence are orthonormal.
In  practice there is a severe loss of orthogonality and a form of
reorthogonalization (Line 6 in Algorithm 2) is necessary.

Let $\theta_i, \ i=1\cdots, m$ be  the eigenvalues of $T_m$. These are
termed \emph{Ritz  values}. If $\{ y_i\}_{i=1:m}$,  are the associated
eigenvectors,  then  the vectors  $\{  V_m  y_i\}_{i=1:m}$ are  termed
\emph{Ritz  vectors}  and  they  represent  corresponding  approximate
eigenvectors of  $A$.  Typically, eigenvalues  of $A$ on both  ends of the
spectrum are first well approximated by corresponding  eigenvalues of $T_m$
(Ritz values) and, as more steps are taken, 
more and more eigenvalues toward
the inside of the spectrum become better approximations. 
Thus, one can say that the Ritz values approximate the eigenvalues of $A$ 
progressively from `outside in'.

%%-----------------------------------------------------------------------

One approach to compute the DOS is to  compute
 these $\theta_i$'s  and then get approximate DOS from them.
However, the $\theta_i$'s tend to provide poor approximations
to  the eigenvalues located at the  interior of the spectrum
 and so this approach does not work too well in practice.
A better idea is to exploit the relation between the Lanczos
procedure and the  (discrete)
orthogonal polynomials and the related Gaussian quadrature.

Assume the initial vector $v_1$ in the Lanczos method can be expanded in the eigenbasis of $A$ as $v_1 = \sum_{i=1}^{n}\omega_iu_i$.
Then the  Lanczos process builds orthogonal polynomials with respect to the discrete (Stieljes)
inner product:
\eq{eq:InProd} 
\int_{a}^{b} f(t) q(t) d\mu(t) \equiv (f(A)v_1, q(A)v_1),
\en 
where the measure $\mu(t)$ is a piecewise constant function defined as
  \begin{equation}
    \mu(t)=
    \begin{cases}
      0, & \text{if}\ t\!<\!a =\lambda_1, \\
      \sum_{j=1}^{i-1} \omega_j^2, & \text{if}\ \lambda_{i-1}\leq t\!<\!\lambda_i, \ i = 2:n,\\
      \sum_{j=1}^n \omega_j^2, & \text{if}\ b = \lambda_n\leq t.
    \end{cases}
  \end{equation}

In particular, when $q(t) = 1$, (\ref{eq:InProd}) takes the form of
\eq{eq:StInt} 
\int_{a}^{b} f(t) d\mu(t) \equiv (f(A)v_1, v_1) , 
\en 
which we will refer to as the Stieljes integral of $f$.
Golub and Welsh \cite{golubwel} showed how to 
extract Gaussian-quadrature formulas for integrals of the type
shown above. The integration nodes for 
a Gaussian quadrature formula with $m$ points, are simply the eigenvalue values $\theta_i, i=1,\cdots,m$  of $T_m$. The associated
weights are the squares of the first components of the eigenvectors
associated with $\theta_i$'s. Thus,
\eq{eq:GW}
\int_{a}^{b} f(t) d\mu(t) \approx \sum_{i=1}^m a_i f(\theta_i) \ ,   \quad a_i = 
\left[ e_1^T  y_i \right]^2 . 
\en
As is known, such an integration formula is  exact for polynomials of
 degree up to $2m-1$, see, e.g., \cite{Golub94matrices,golubwel}. Then we will derive an approximation to the DOS with the quadrature rule (\ref{eq:GW}).

The Stieljes integral 
$\int_{a}^{b} f(t) d\mu(t) $ satisfies the following equality:
\[
\int_{a}^{b} f(t) d\mu(t) = (f(A)v_1,v_1) = \sum_{i=1}^{n} \omega_i^2  
f(\lambda_i).
\]
We can view this  as a distribution $\phi_{v_1} $ applied to $f$:
\eq{eq:LandDis} (f(A) v_1, v_1) \equiv \left\langle \phi_{v_1}, f \right\rangle
 \quad  \mbox{with} \quad 
\phi_{v_1} \equiv \sum_{i=1}^{n} \omega_i^2 \delta (t- \lambda_i) . 
\en
Assume for a moment that we are able to find a special vector $v_1$ which satisfies $\omega_i^2 = 1/n$ for all $i$. Then
the above distribution becomes
$\phi_{v_1}=\frac{1}{n}\sum_{i=1}^{n} \delta(t-\lambda_i) $ which is exactly the DOS defined in (\ref{eq:DOS1}).
 Next,
 we consider how $\phi_{v_1}$ can be approximated via
Gaussian-quadrature. 
%First we have
%\[ 
%\left\langle \phi_{v_1}, f \right\rangle \equiv (f(A)v_1,v_1) 
%= \sum \gamma_i^2 f(\lambda_i). 
%\] 
Based on (\ref{eq:GW}) and (\ref{eq:LandDis}), we know that
\[
\left\langle \phi_{v_1}, f \right\rangle \equiv (f(A)v_1,v_1) = \int_{a}^{b} f(t) d\mu(t)\approx \sum_{i=1}^{m} a_i f(\theta_i) 
\equiv 
\left\langle\sum_{i=1}^{m} a_i  \delta(t-\theta_i), f \right\rangle. 
\]
Since $f$ is an arbitrary function, we obtain the following approximation expressed for the DOS:
\eq{eq:LanQuadApp}
\phi_{v_1}  
\approx \tilde{\phi}_{v_1}:=\sum_{i=1}^{m} a_i \delta(t-\theta_i) .
\en

In the next theorem, we show that the approximation error of the Lanczos method for computing the DOS decays as $\rho^{-2m}$ for a constant $\rho \!>\!1$. Here, we follow (2.5) in \cite{DOSpaper16} to measure the approximation error between $\phi$ and $\tilde{\phi}_{v_1}$ as
\[
|\left\langle \phi, g \right\rangle - \left\langle \tilde{\phi}_{v_1}, g \right\rangle|, \quad \text{with} \ g(t) \ \text{being an analytic function on} \ [-1, 1].
\]
 \begin{theorem}\label{thm:errorLanczos}
   Assume $A  \in \mathbb{C}^{n\times n}$  is a Hermitian  matrix with
   its spectrum  inside $[-1, \  1]$. If $v_1\in \mathbb{R}^{n}$  is a
   unit vector with equal  weights in all eigenvectors of $A$,
   then   the    approximation   error    of   a    m-term   expansion
   (\ref{eq:LanQuadApp}) is
\begin{equation}
\label{eq:lanczosb}
|\left\langle \phi, g \right\rangle - \left\langle \tilde{\phi}_{v_1}, g \right\rangle| \leq \frac{4\rho^2M(\rho)}{(\rho^2-1)\rho^{2m}},
\end{equation}
where $\rho\!>\!1$ and $M(\rho)$ are constants.
\end{theorem}  
\begin{proof}
Let $p_{2m-1}$ be the Chebyshev polynomial approximation of degree  $2m-1$  to $g(t)$:
\[
p_{2m-1}(t) = \sum_{k=0}^{2m-1}\gamma_kT_k(t)\approx g(t) = \sum_{k=0}^{\infty}\gamma_kT_k(t).
\]
Since the quadrature formula (\ref{eq:GW}) is exact for  polynomials with degree up to $2m-1$, we have
\[
\int_{a}^{b} p_{2m-1}(t) d\mu(t) = \sum_{i=1}^m a_i p_{2m-1}(\theta_i). 
\]
Therefore, we get
\begin{align*}
|\left\langle \phi, g \right\rangle - \left\langle \tilde{\phi}_{v_1}, g \right\rangle| &= |\frac{1}{n}\sum_{i=1}^{n}g(\lambda_i) - \sum_{j=1}^{m}a_jg(\theta_j)|= | \int_{a}^{b} g(t) d\mu(t) - \sum_{j=1}^{m}a_jg(\theta_j)|\\
&\leq | \int_{a}^{b} g(t)- p_{2m-1}(t) d\mu(t)|+| \int_{a}^{b} p_{2m-1}(t)d\mu(t)-\sum_{j=1}^{m}a_jg(\theta_j)|\\
& \leq \int_{a}^{b} |g(t)- p_{2m-1}(t)| d\mu(t)+\sum_{j=1}^{m}a_j|p_{2m-1}(\theta_j)-g(\theta_j)|\\
&\leq  \sum_{k=2m}^{\infty} \int_{a}^{b} |\gamma_k| |T_k(t)| d\mu(t) + \sum_{j=1}^{m}a_j\sum_{k=2m}^{\infty}|\gamma_k||T_k(\theta_j)|.
\end{align*}
Based on (\ref{eq:trf4}), we know that 
\begin{align*}
\sum_{k=2m}^{\infty}|\gamma_k||T_k(\theta_j)| \leq \sum_{k=2m}^{\infty} 2 M(\rho) \rho^{-k}|T_k(\theta_j)|\leq \sum_{k=2m}^{\infty} 2 M(\rho) \rho^{-k}.
\end{align*}
Since $\sum_{j=1}^{m} a_{j} = \int_{a}^{b} d\mu(t) = (v_{1},v_{1}) = 1$, we have 
\begin{equation}\label{eq:bd1}
\sum_{j=1}^{m}a_j\sum_{k=2m}^{\infty}|\gamma_k||T_k(\theta_j)| \leq \frac{2\rho^2M(\rho)}{(\rho^2-1)\rho^{2m}}\sum_{j=1}^{m} a_{j}= \frac{2\rho^2M(\rho)}{(\rho^2-1)\rho^{2m}}.
\end{equation}

%% -----------------------------------------------------------------------
%% %% YS : issues here. Corrected below
%% 
%% 
%% On the other hand, by the min-max principle for Herimitian eigenvalue problems, we know  
%% \[
%% \int_{a}^{b} |T_k(t)| d\mu(t) = v_{1}^{T}|T_{k}(A)|v_{1}\leq \max \lambda(|T_k(A)|) \leq 1.
%% \]
%% 
%% -----------------------------------------------------------------------
%% 
For the first term, we have 
\[
\int_{a}^{b} |T_k(t)| d\mu(t) =  \frac{1}{n} 
\sum_j |T_k(\lambda_j)| \le 1 ,
\]
and therefore, 
\begin{equation}\label{eq:bd2}
\sum_{k=2m}^{\infty} \int_{a}^{b} |\gamma_k| |T_k(t)| d\mu(t)   \leq  \sum_{k=2m}^{\infty} 2 M(\rho) \rho^{-k} \leq \frac{2\rho^2M(\rho)}{(\rho^2-1)\rho^{2m}}.
\end{equation}
Adding the bounds in (\ref{eq:bd1}) and (\ref{eq:bd2}), we obtain  (\ref{eq:lanczosb}).
\end{proof}

Theorem \ref{thm:errorLanczos} indicates that the approximation error from the Lanczos method for computing the DOS decays as $\rho^{-2m}$, which is twice as fast as the KPM method with degree $m$.

%In the end our wanted distribution is simply approximated by 
%$\sum_{i=1}^{m} a_i \delta(t-\theta_i)$. 
The approximation in (\ref{eq:LanQuadApp}) is achieved by taking an
idealistic vector $v_1$ that has equal weights ($\pm 1/\sqrt{n}$)  in all eigenvectors
in its representation in the eigenbasis. 
A common strategy to mimic the effect of having a vector with 
$\mu_i=1/\sqrt{n},\forall i$,  is to use $s$  random 
vectors $v_1\up{k}$, called sample vectors, and average the results of the 
above formula over them:
\eq{eq:LanDos}
\phi 
\approx 
\frac{1}{s} \sum_{k=1}^s 
\sum_{i=1}^m a_i\up{k} \delta(t -{\theta_i\up{k}}) .
\en
Here the superscript $(k)$ relates to the $k$-th sample vector and 
$\theta_i\up{k}$, $ a_i\up{k}$ are the nodes and  weights of the
quadrature formula shown in \nref{eq:GW} for this sample vector. 

\subsection{Generalized problems} 
A straightforward way to deal with the generalized case is to apply 
the standard Lanczos algorithm (Algorithm 2)
described in the previous section 
to the matrix $S \inv A S\inv $ (or $L\inv A L^{-T}$). This leads to the relation:
\begin{equation}\label{eq:LanStd} 
S\inv A S\inv V_m = V_m T_m + \beta_{m+1} v_{m+1} e_m \trans. 
\end{equation}
If we set  $W_m=S\inv V_m$, and multiply through by $S\inv$, then we get
\begin{equation} \label{eq:LanB} 
B\inv A W_m = W_m T_m + \beta_{m+1} w_{m+1} e_m \trans,
\end{equation}
where it is important to note that $W_m$ is $B$-orthogonal since 
\[
W_m \trans B W_m = V_m \trans S\inv B S\inv V_m = V_m \trans V_m = I.
\]
%and that the matrix $B\inv A$ is self-adjoint with respect to the 
%$B$-inner product. 
It  is possible  to  generate a  basis $V_m$  of  the Krylov  subspace
$K_m(v_1,S\inv A S\inv)$ if we want  to deal with the standard problem
with $S\inv A S\inv$. It is  also possible to generate the basis $W_m$
of the Krylov subspace $K_m(w_1,B\inv A)$  directly if we want to deal
with  the  standard  problem  with   $B\inv  A$  using  the  $B$-inner
product. From our  discussion at the end  of Section \ref{sec:kpmgen},
we know that the second case is computationally more efficient.

%% YS: this was hanging here... no
%% Alternatively, we can set  $W_m=L^{-T} V_m$ and get
%% \begin{equation} \label{eq:LanBL} 
%% B\inv A W_m = W_m T_m + \beta_{m+1} w_{m+1} e_m \trans,
%% \end{equation}

%Case one is straightforward and it involves
%the Cholesky factorization of $A$ or, %alternatively, its square root. 

Now let us focus on the case (\ref{eq:LanB}). If we
start the Lanczos algorithm with a vector $w_1$ where 
$\norm{w_1}_B=1$, we could generate the sequence $w_i$ through Algorithm 3, which is described as
Algorithm 9.2 in \cite[p.230]{Saad2011}.
\begin{algorithm}[H]
\begin{algorithmic}[1]  \label{alg:LanAB}
\caption{Lanczos algorithm for matrix pair $(A,B)$}
\STATE Choose an initial vector $w_1$ with $\norm{w_1}_B=1$. Set $\beta_1=0$, $w_0=0$, $z_0=0$, 
       and compute $z_1=Bw_1$
\FOR {$j=1,2,\ldots,m$}
\STATE $z:=Aw_j-\beta_j z_{j-1}$
\STATE $\alpha_j=(z,w_j)$
\STATE $z:=z-\alpha_j z_j$
\STATE Full reorthogonalization: $z:=z-\sum_i(z,w_i)z_i$ for $i\le j$
\STATE $w:=B\inv z$
\STATE $\beta_{j+1}=\sqrt{(w,z)}$
\STATE If $\beta_{j+1}==0$ restart or exit 
\STATE $w_{j+1}:= w / \beta_{j+1}$
\STATE $z_{j+1}:= z / \beta_{j+1}$
\ENDFOR
\end{algorithmic}
\end{algorithm}

It is easy to show that if we set $v_i = S w_i$, then 
the $v_i$'s are orthogonal to each other and that
they are identical with the sequence of $v_i$'s 
that would be obtained from the standard Lanczos algorithm applied to
$S\inv A S\inv $ (or $L^{-1} A L^{-T}$) starting with $v_1 = S w_1$ (or $v_1 = L^{-T} w_1$). The two algorithms are equivalent and going from one to
other requires a simple transformation.

The  3-term recurrence now becomes
\begin{equation}
\beta_{m+1} w_{m+1} = \hat w_{m+1} = B\inv A w_m -  \alpha_m w_m - \beta_m w_{m-1},   
\end{equation}
and $\beta_{m+1}=(B \hat w_{m+1}, \hat w_{m+1})^{1/2}$. 
Note that the algorithm requires that we  save the auxiliary
 sequence $z_j \equiv B w_j$ in order to avoid additional computations
with $B$ to calculate $B$-inner products.

%Assume that we want to use this second form of the Lanczos algorithm 
%to calculate the DOS. 
On the surface the extension seems trivial: we could
take a sequence of random vectors $w_1\up{k}$ and compute an
average  analogue to \nref{eq:LanDos} over these vectors.
There is  a problem in the selection of the initial 
vectors. We can reason with respect to the original algorithm applied to
$S\inv A S\inv$. 
If we take a random vector $v_1\up{k} $ and run Algorithm 2 with this as a starting vector, 
 we would compute the exact same tridiagonal matrix
$T_m\up{k}$ as if we used  Algorithm 3 with
$w_1\up{k} = S\inv v_1\up{k}$. Using  the same
average \nref{eq:LanDos} appears therefore perfectly valid since
the corresponding $\theta_i\up{k}$ and $a_i\up{k}$ are the same.  
The catch is in the way we select the initial vectors $w_1\up{k}$.
Indeed, it is not enough to select random vectors $w_1\up{k}$ 
 with mean zero and variance $1/n$, \emph{it is the associated
$v_1\up{k}$ that should have this property.} 
Selecting $w_1\up{k} $ to be of mean zero and variance
$1/n$, 
will not work, since the corresponding $v_1\up{k} \equiv S w_1\up{k}$ 
will have mean zero but not the right variance. 

The only modification that is implied by this observation is that
we will need to modify the initial step of Algorithm 3
as follows:

\medskip  
1.\  \parbox[t]{0.91\textwidth}
{Choose $v_1$ with components $ \eta_i \in {\cal N}(0,1)$ and let
$w_1 = S^{-1} v_1$ (or $w_1 = L^{-T} v_1$); $z_1 = B w_1$. Compute $t = \sqrt{ (w_1, z_1)}$
and $z_1:=z_1/t; w:=w_1/t$. Set $\beta_1= 0; z_0 =w_0 = 0$.} 
 
%Note that $\|w_1\|_B = 1 $ and that $z_1 = B w_1$ as is required in
%initialization of Algorithm 3. The only difference 
%is that, apart from a multiplicative constant, $w_1 = q_m(B) v_1$ where
%$v_1$ is a random vector. 
%
%Thus, solving with the Cholesky factor $L^T$, 
% or applying the inverse square $B^{-1/2}$ is required at least 
%once at the start of the procedure.
%As was seen in Section~\ref{sec:funapp} this can inexpensively achieved,
%typically with just a few matrix-vector products involving the matrix 
%$B$. This difference between KPM and the Lanczos method for 
%computing the DOS is remarkable. The KPM essentially relies only
%on estimating the trace and, as was seen earlier, this still 
%works for matrices of the type $B\inv A$. The Lanczos procedure
%on the other hand, is based on Gaussian quadrature ...
%
%-------------------- TBC -------------------->

\section{Numerical Experiments} \label{sec:numerical}  In this section
we  illustrate the  performance of  the  KPM and  Lanczos methods  for
computing the DOS for generalized eigenvalue problems. Both algorithms
have been implemented in MATLAB and all the experiments were performed
on a Macbook Pro with Intel i7 CPU processor and 8 GB memory.

In  order  to  compare  with  the  accuracy  of  the  DOS,  the  exact
eigenvalues of each problem are computed with MATLAB built-in function
\textbf{eig}. We  measure the error  of the approximate DOS  using the
relative $L_1$ error as proposed in \cite{Lin2017}:
\begin{equation}
\text{ERROR} = \frac{\sum_i|\tilde{\phi}_\sigma(t_i)-\phi_\sigma(t_i)|}{\sum_i|\phi_\sigma(t_i)|},
\end{equation}
where $\{t_i\}$ are a set of uniformly distributed points and $\phi_\sigma(\cdot)$ and $\tilde{\phi}_\sigma(\cdot)$ are the smoothed (or regularized) DOS with $\delta(t)$ replaced by $\frac{1}{\sqrt{2\pi}\sigma}e^{\frac{-t^2}{2\sigma^2}}$. A heuristic criterion to select $\sigma$ as suggested in \cite{DOSpaper16} is to set
\begin{equation}
\sigma = \frac{\lambda_{max}-\lambda_{min}}{60\sqrt{2\log (1.25)}},
\end{equation} 
where $\lambda_{max}$ and $\lambda_{min}$ are the largest and smallest eigenvalues of the matrix pencil $(A,B)$. 

\subsection{An example from the earth's normal mode simulation}
The first example  is from the study of the  earth's normal modes with
constant solid materials. The stiffness matrix $A$ and mass matrix $B$
 result  from the continuous  Galerkin finite element  method and
have size  of $n = 3,657$.  Details about the  physical model and
the discretization techniques used can be found in \cite{shigmigreport2016note,normalmodes}.

The  numbers of  nonzero entries  are $A$  and $B$  are $145,899$  and
$48,633$, respectively. The eigenvalues of the pencil are ranging from
$\lambda_{min}  =   -2.7395\!\times\!10^{-13}$  to   $\lambda_{max}  =
0.0325$. Fig.~\ref{fig:nm} displays the  sparsity patterns of $A,B$ as
well as the histogram of the eigenvalues of $(A,B)$.

\begin{figure}[htb] 
\begin{tabular}
[c]{ccc}%
\includegraphics[width=0.29\textwidth]{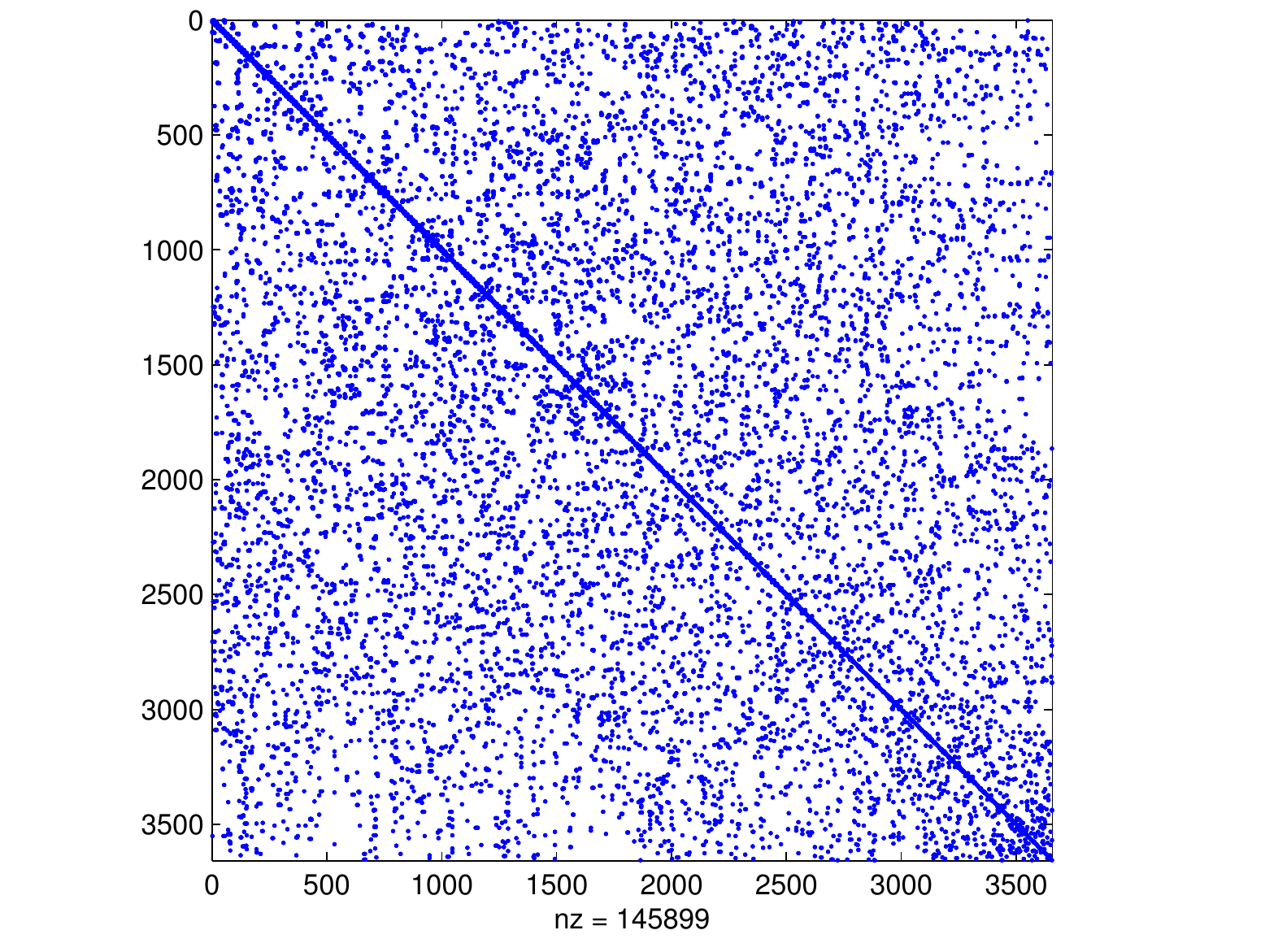}&
\includegraphics[width=0.29\textwidth]{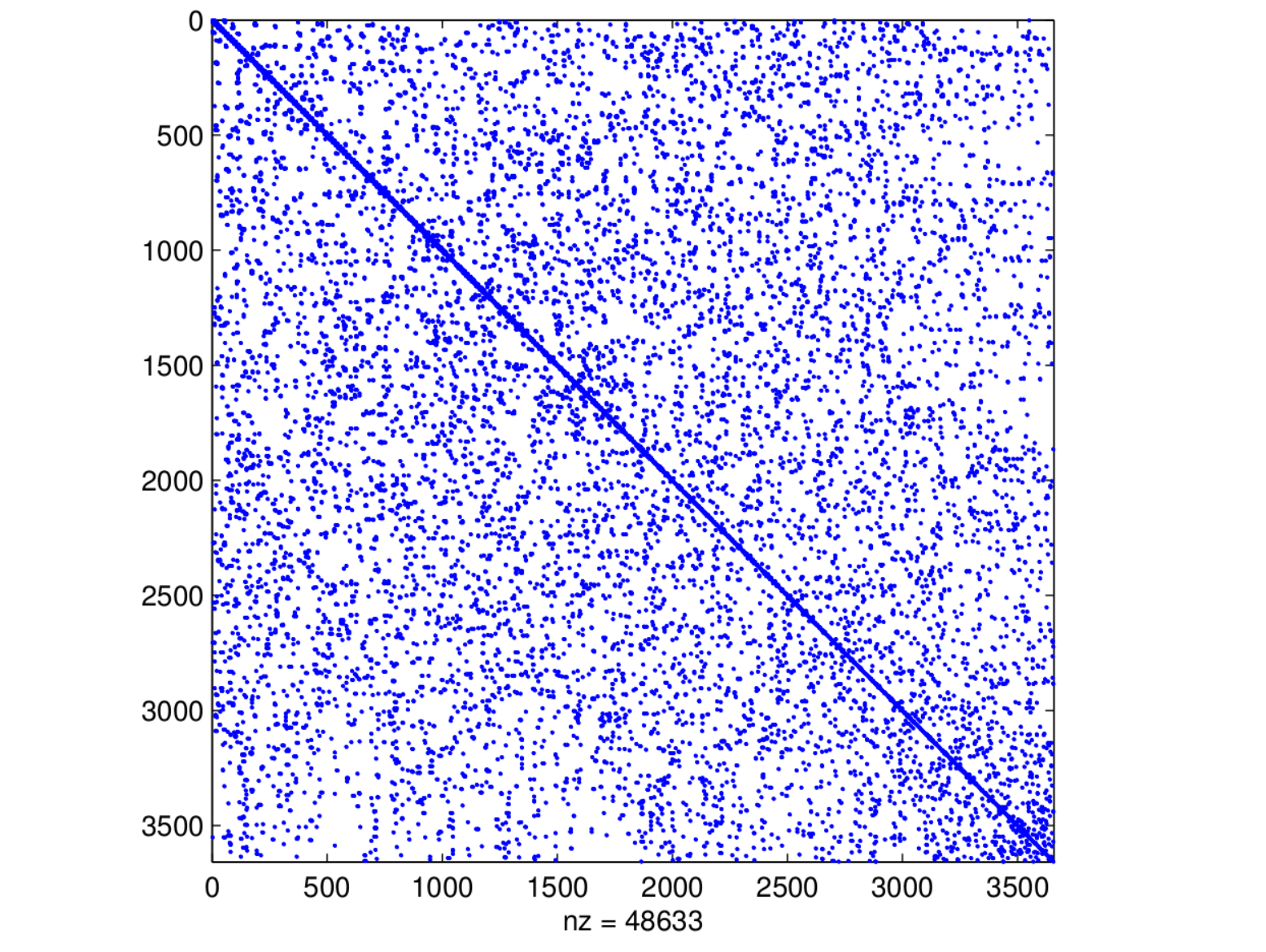}&
\includegraphics[width=0.29\textwidth]{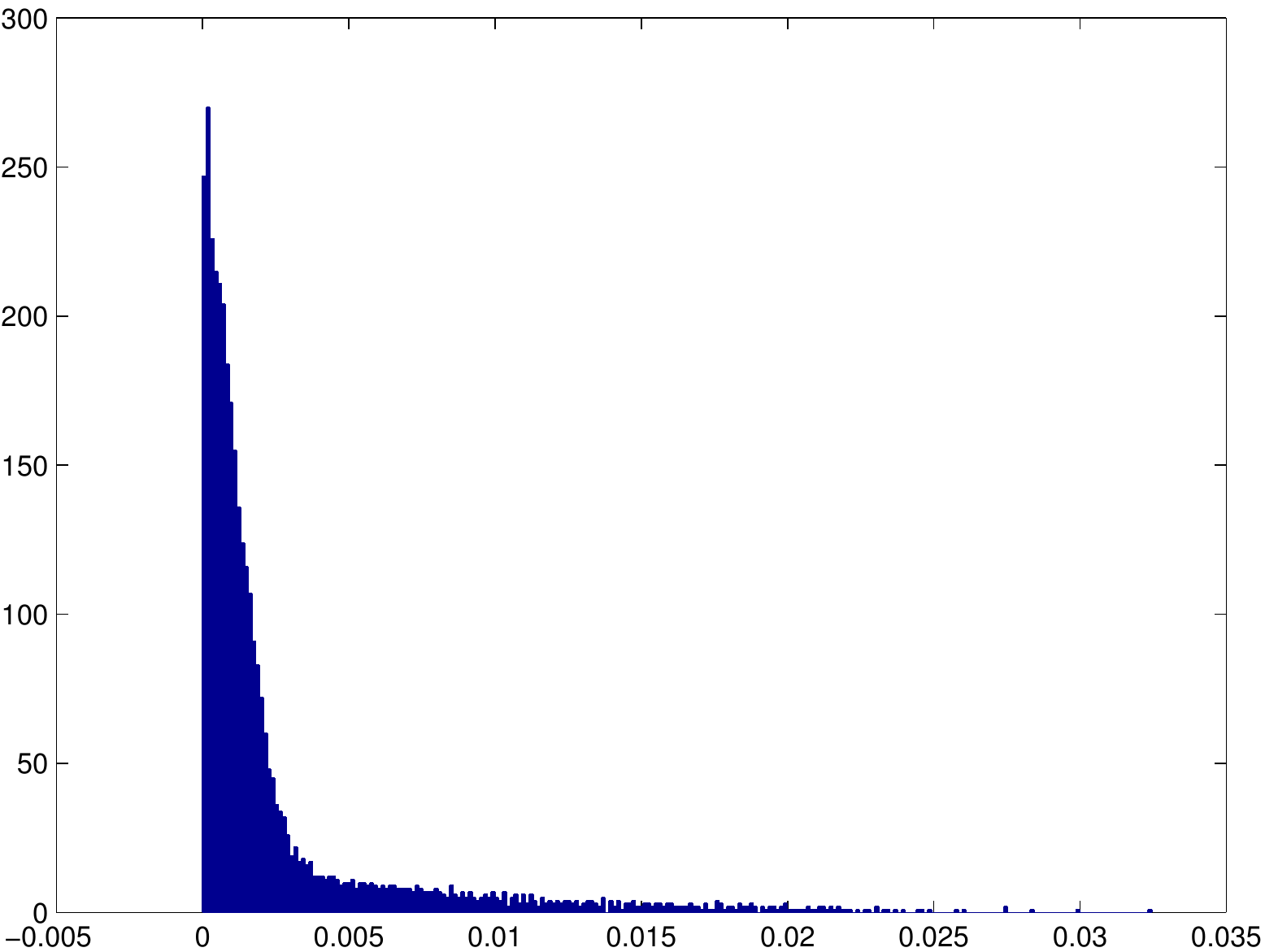}\\{\small (i) Sparsity of $A$} & {\small (ii) Sparsity of $B$} & {\small (iii) Histogram of eigenvalues}\\
\end{tabular}
\caption{For the earth's normal mode matrix pencil, the sparsity pattern of $A$, $B$ and the histogram of the eigenvalues of $(A,B)$ with $250$ bins.}
\label{fig:nm}
\end{figure}

In Fig.~\ref{fig:lanvskpm}, we first  compare the computed accuracy of
the KPM  with that  of the  Lanczos method when  the number  of random
vector $n_{nev}$ was fixed at  $50$. The Cholesky factorization of $B$
was used  for operations  involving $B$. We  observe that  the Lanczos
method outperforms the KPM when $m$  varies from $20$ to $60$. This is
because the  eigenvalues of  this pencil are  clustered near  the left
endpoint of  the spectrum  (See Fig.~\ref{fig:nm}  (iii)) and  the KPM
method has a  hard time capturing this cluster (See Fig.~\ref{fig:lanvskpm2}).

\begin{figure}[htb] 
\centerline{\includegraphics[width=0.6\textwidth]{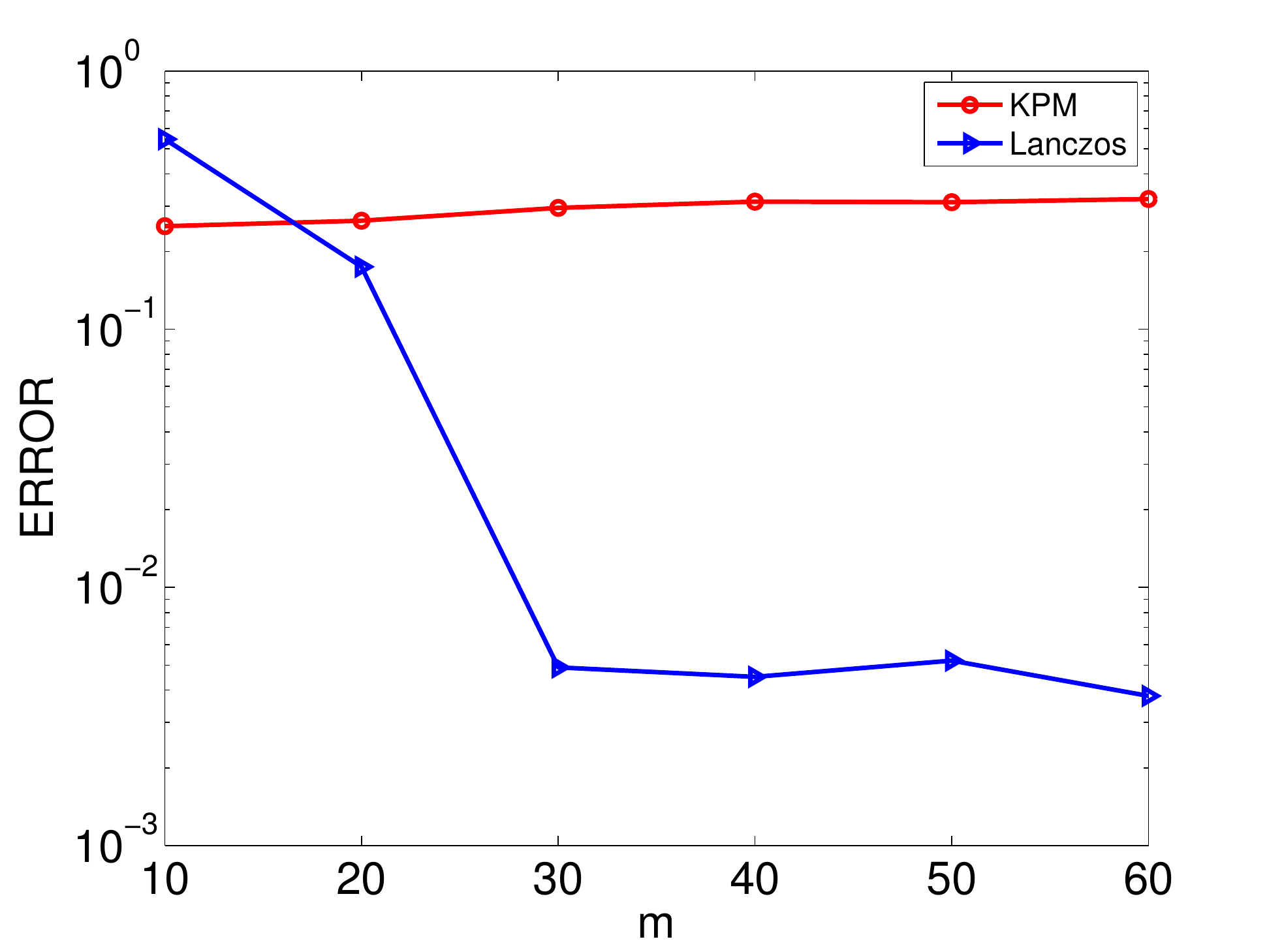}}
\caption{A comparison of  approximation errors of the  KPM and Lanczos
  method  applied  to  the  earth's  normal  mode  matrix  pencil  for
  different $m$  values.  The Cholesky factorization  is performed for
  operations involving $B$ and $n_{nev}$ is fixed at $50$.}
\label{fig:lanvskpm}
\end{figure}

\begin{figure}[htb] 
\includegraphics[width=0.49\textwidth]{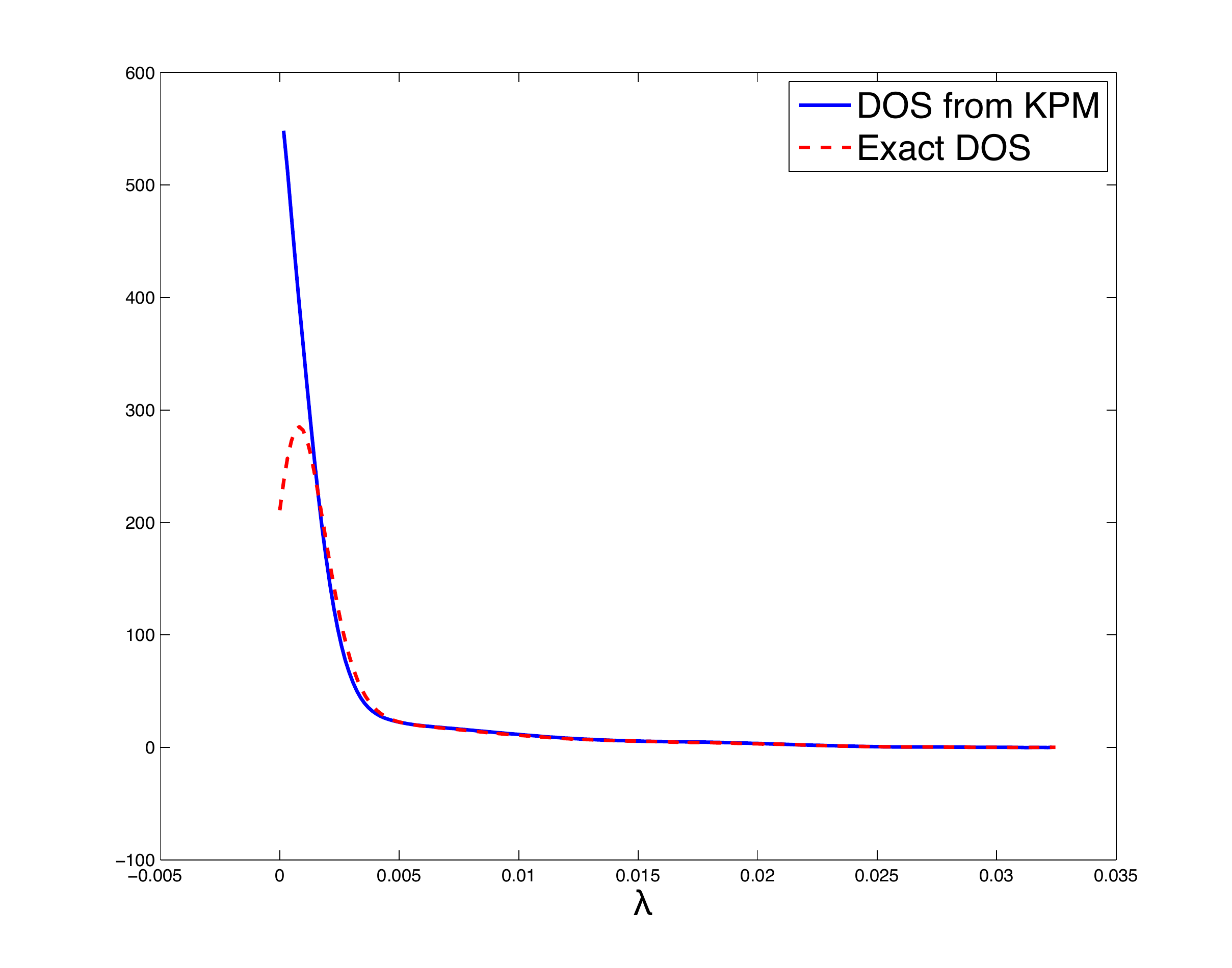}
\includegraphics[width=0.49\textwidth]{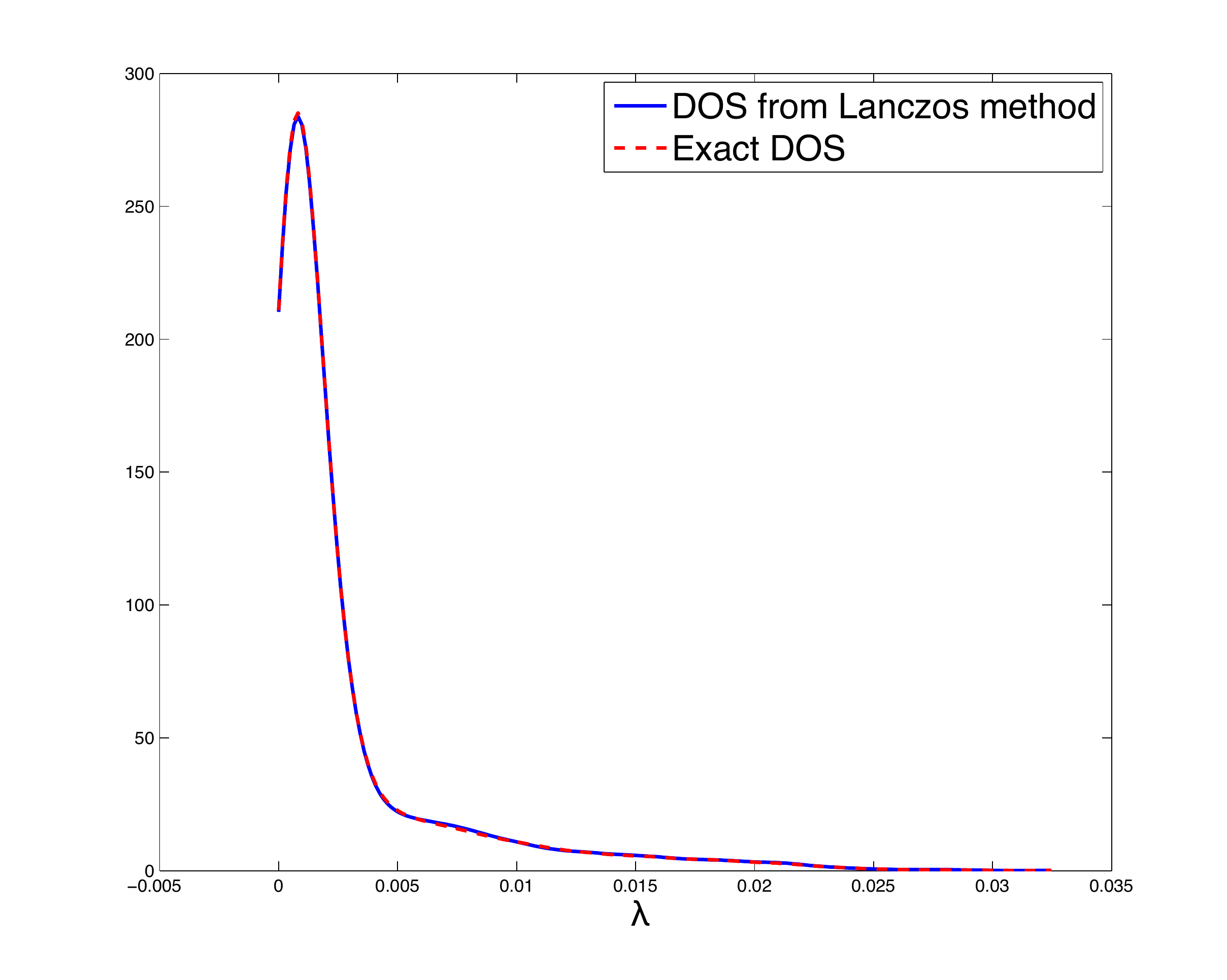}
\caption{For the earth's normal mode matrix pencil, the computed DOS by the KPM (left) 
and the Lanczos method (right) when $m=30$ and $n_{nev}=50$, compared to the exact DOS.
Cholesky factorization is performed for operations involving $B$.}
\label{fig:lanvskpm2}
\end{figure}

Fig.~\ref{fig:lanerror} shows the error of the Lanczos method with an increasing number of random vectors $n_{nev}$ and fixed $m=30$. It indicates that a large number of $n_{nev}$ helps reduce the error through the randomization. 

\begin{figure}[htb] 
\centerline{\includegraphics[width=0.6\textwidth]{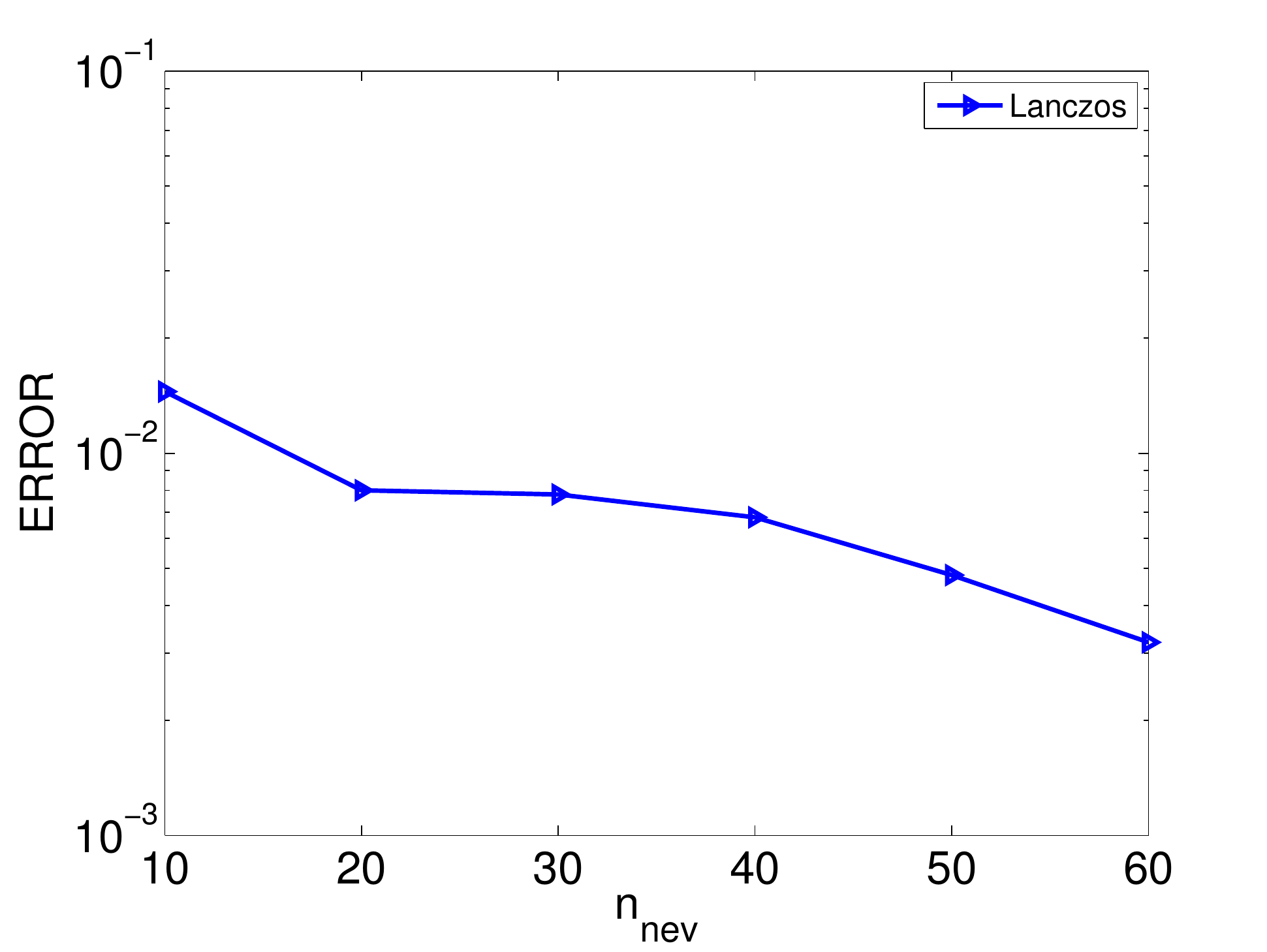}}
\caption{For the earth's normal mode pencil, the error of the Lanczos method with respect to an increasing number of random vectors $n_{nev}$ when $m=30$. Cholesky factorization is performed for operations involving $B$.}
\label{fig:lanerror}
\end{figure}

Then we consider replacing the Cholesky factorization of $B$ with Chebyshev polynomial approximations $g_{k_1}(B)$ and $q_{k_2}(B)$ as proposed in Section \ref{sec:funapp}. One way to determine the degree of $g_{k_1}$ (or $g_{k_2}$) is to use the theoretical result of Theorem~\ref{thm:Main}. However, the theorem has a 
parameter $\rho$ which is free and the selection of optimal $\rho$ may be harder
than the selection of $k_i$ by simpler means. Since $g(t)$, $q(t)$ and their approximations are smooth and a simple heuristic is to select 
$k_i$ to be the smallest number for which the \emph{computed} 
$\| (g - g_{k_1})/g \|_\infty $ and $\| (q - q_{k_2})/q \|_\infty $ are small enough. To evaluate the norm we can 
discretize the interval under consideration very finely (higher degrees
will require more points). This will yield an estimate rather than an exact norm and this is enough for practical purposes. 

For the original matrix pencil $(A,B)$, the eigenvalues of $B$ are inside $[3.80e\!+\!07, 1.46e\!+\!10]$ and $\kappa(B)= 382.91$. In this case, we can estimate the convergence based on $ \rho_1 = [\sqrt{\kappa +1 } + \sqrt{2}]/[\sqrt{\kappa -1 }]= 1.0750$. Since $\rho_1$ is close to $1$, one should expect a slow convergence for $g_{k_1}(B)$ (or $q_{k_2}(B)$) to $B^{-1}$ (or $B^{-1/2}$). In Table \ref{tab:errSqrt}, we report the computed norms $\| (g - g_{k})/g \|_\infty$ and $\| (q - q_{k})/q \|_\infty $ when $k$ increases from $30$ to $60$. As we can see, the error associated with $g_{k}$ is larger than $10^{\!-\!3}$ even when $k$ reaches $60$.  

\begin{table}[htb] 
\begin{center} 
\begin{tabular}{c|c|c} 
Degree $k$ & $\| (g - g_{k})/g  \|_\infty$ & $\| (q - q_{k})/q  \|_\infty $\\ \hline 
30 &  $8.62\!\times\! 10^{-1}$ &  $1.92\!\times\! 10^{-2}$\\
40 &  $3.10\!\times\! 10^{-1}$ &  $6.00\!\times\! 10^{-3}$ \\
50 &  $1.12\!\times\! 10^{-1}$ &  $2.00\!\times\! 10^{-3}$\\
60 &  $4.01\!\times\! 10^{-2}$ &  $6.45\!\times \!10^{-4}$ \\
\hline
\end{tabular}
\end{center} 
\caption{Computed error norms for the order $k$ Chebyshev polynomial approximations to $g\!=\!1/\lambda$ and $q\!=\!1/\sqrt{\lambda}$ on the interval $[3.8017\!\times\!10^{7},1.4557\!\times\!10^{10}]$, which contains the spectrum of the original mass matrix $B$.}\label{tab:errSqrt}
\end{table}

We then applied the diagonal scaling technique to the mass matrix $B$. The eigenvalues of $D^{-1/2}BD^{-1/2}$ are now inside $[0.5479, 2.500]$ and $\kappa(D^{-1/2}BD^{-1/2}) = 4.5629$. In this case, 
$\rho_1 = 1.9988$ and $g_{k_1}(B)$ and $q_{k_2}(B)$ converge much faster. This is confirmed in Table \ref{tab:errSqrt2} where the  error norms are smaller than $6\!\times\! 10^{\!-6\!}$ for both approximations when $k$ reaches 12.
 
\begin{table}[htb] 
\begin{center} 
\begin{tabular}{c|c|c} 
Degree $k$  & $\| (g - g_{k})/g  \|_\infty$ & $\| (q - q_{k})/q  \|_\infty $\\ \hline 
6 &  $2.60\!\times\! 10^{-2}$ &  $3.73\!\times\! 10^{-4}$\\
8 &  $3.36\!\times\! 10^{-4}$ &  $4.32\!\times\! 10^{-5}$\\
10 &  $4.42\!\times\! 10^{-5}$ &  $5.13\!\times\! 10^{-6}$ \\
12 &  $5.80\!\times\! 10^{-6}$ &  $6.19\!\times\! 10^{-7}$\\
\hline
\end{tabular}
\end{center} 
\caption{Computed error norms for the order $k$ Chebyshev polynomial approximations to $g\!=\!1/\lambda$ and $q\!=\!1/\sqrt{\lambda}$ on the interval $[0.5479, 2.500]$, which contains the spectrum of the mass matrix $B$ after diagonal scaling.}\label{tab:errSqrt2}
\end{table}

Fig.~\ref{fig:lan2} shows the error of the Lanczos method when the operations $B^{-1}v$ and $B^{-1/2}v$ are approximated by $g_{k_1}(B)v$ and $q_{k_2}(B)v$, respectively. The number of sample vectors $n_{nev}$ was fixed at $50$ and the degree $m$ was fixed at $30$. The degrees of $g_{k_1}$ and $q_{k_2}$ are determined to be the smallest integers for which the following inequalities hold
\begin{equation}
\| (g - g_{k_1})/g  \|_\infty\leq \tau,\quad 
\| (q - q_{k_2})/q  \|_\infty\leq \tau.
\label{eq:tau}
\end{equation}
Although the exact DOS curve is indistinguishable from those obtained from the Lanczos method, the error actually decreases as we reduce the value of $\tau$. The errors are $1.41\!\times\!10^{-2}$, $5.61\!\times\!10^{-3}$, $4.70\!\times\!10^{-3}$ and $4.30\!\times\!10^{-3}$ when $\tau = 10^{-1},10^{-2},10^{-3}$ and $10^{-4}$, respectively. In the following experiments, we will fix $\tau$ at $10^{-3}$ to select the degree for $g_{k_1}$ and $q_{k_2}$ based on (\ref{eq:tau}).

\begin{figure}[htb] 
\begin{tabular}
[c]{cc}%
\includegraphics[width=0.46\textwidth]{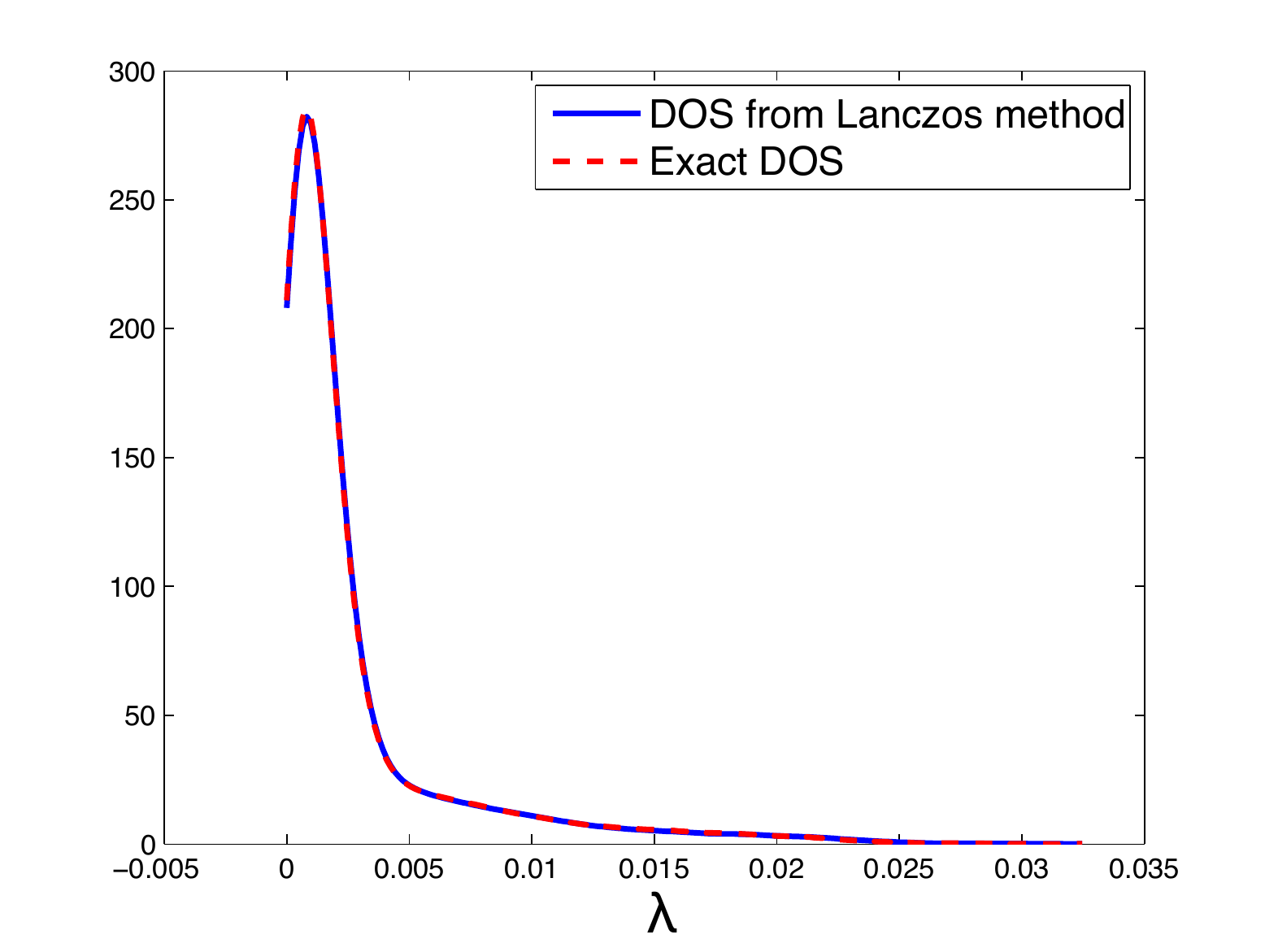}&
\includegraphics[width=0.46\textwidth]{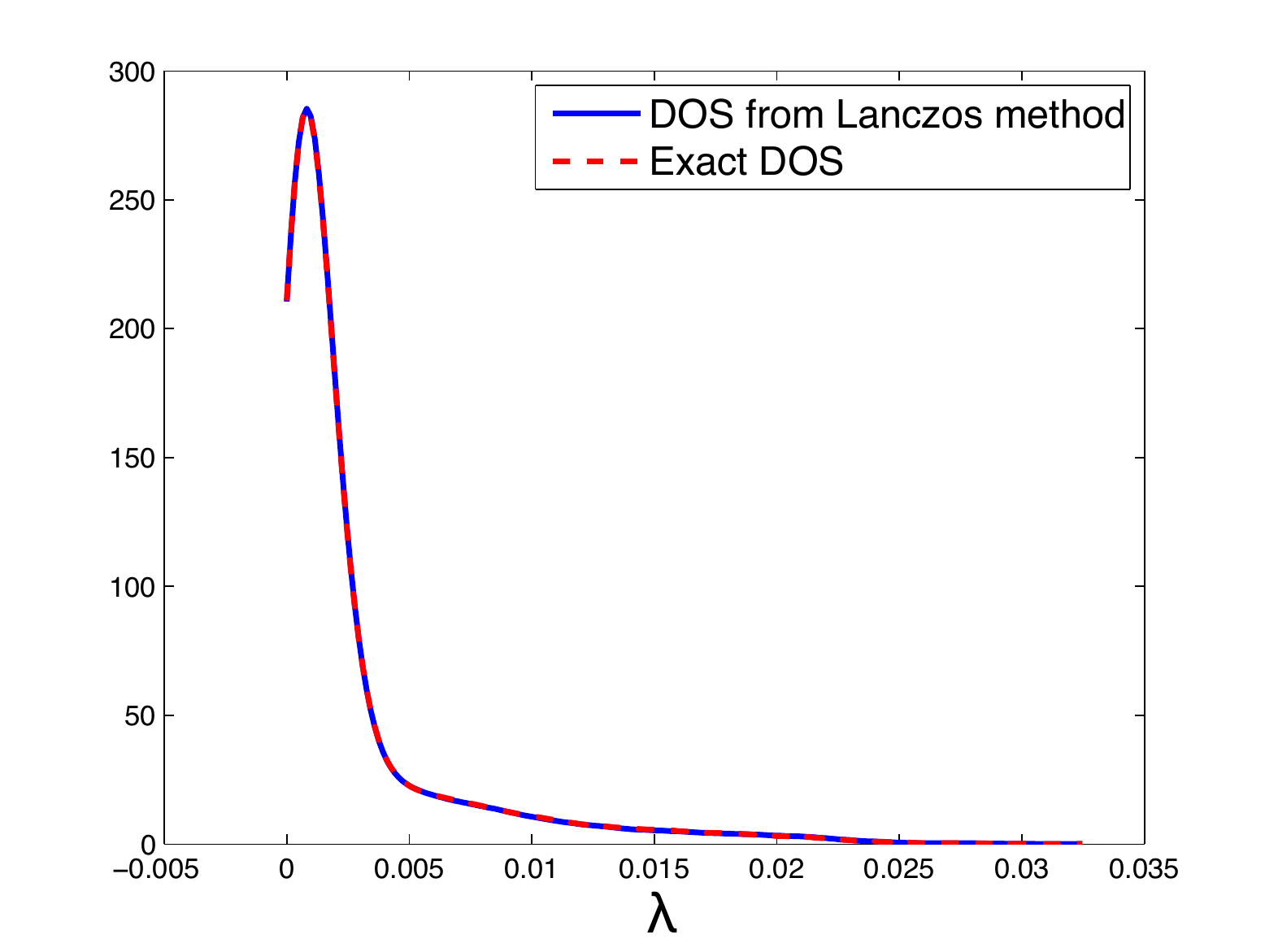}\\
{\small (i) $\tau=10^{-1}$} & {\small (ii) $\tau=10^{-2}$}\\
\includegraphics[width=0.46\textwidth]{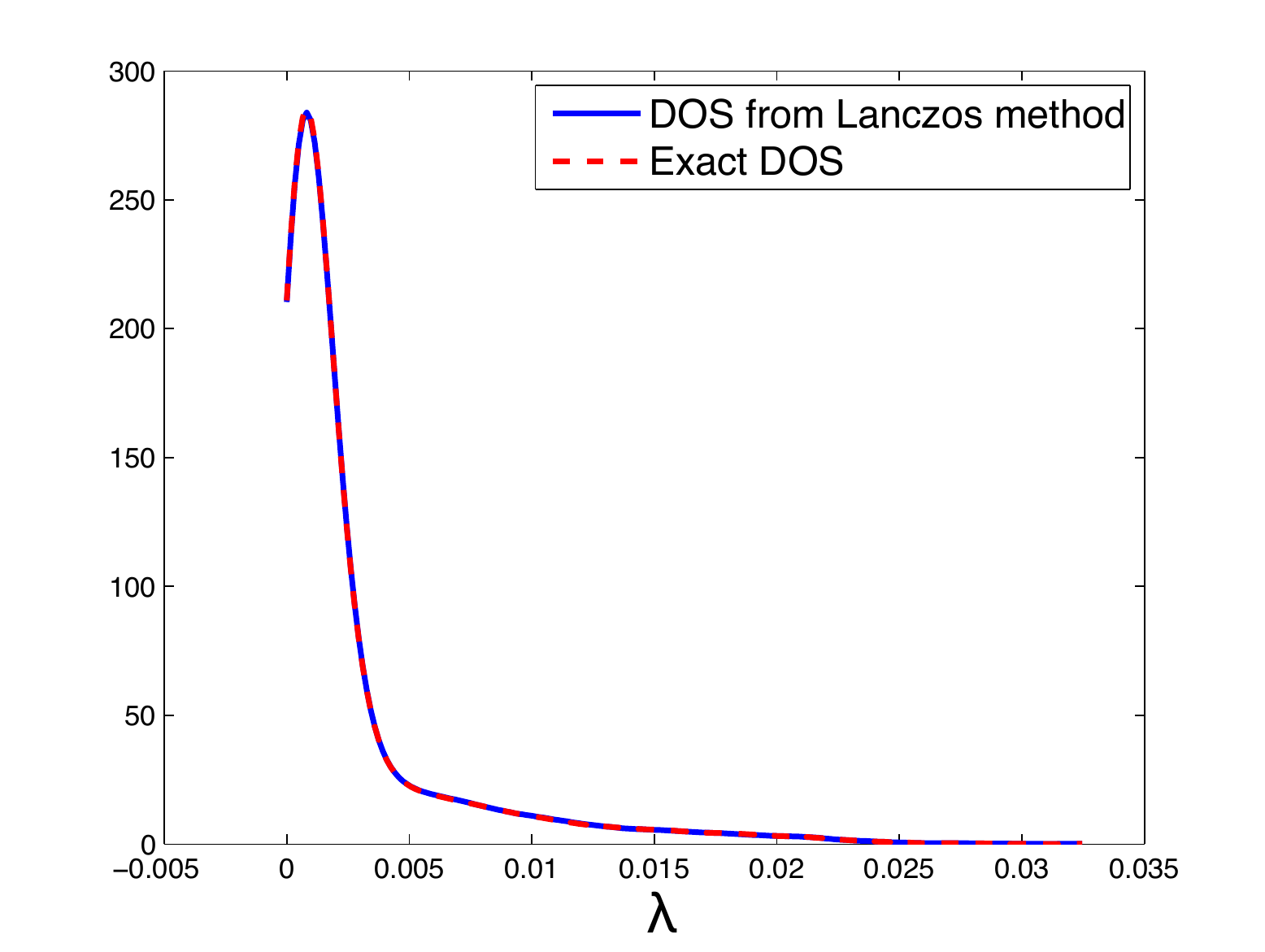}&
\includegraphics[width=0.46\textwidth]{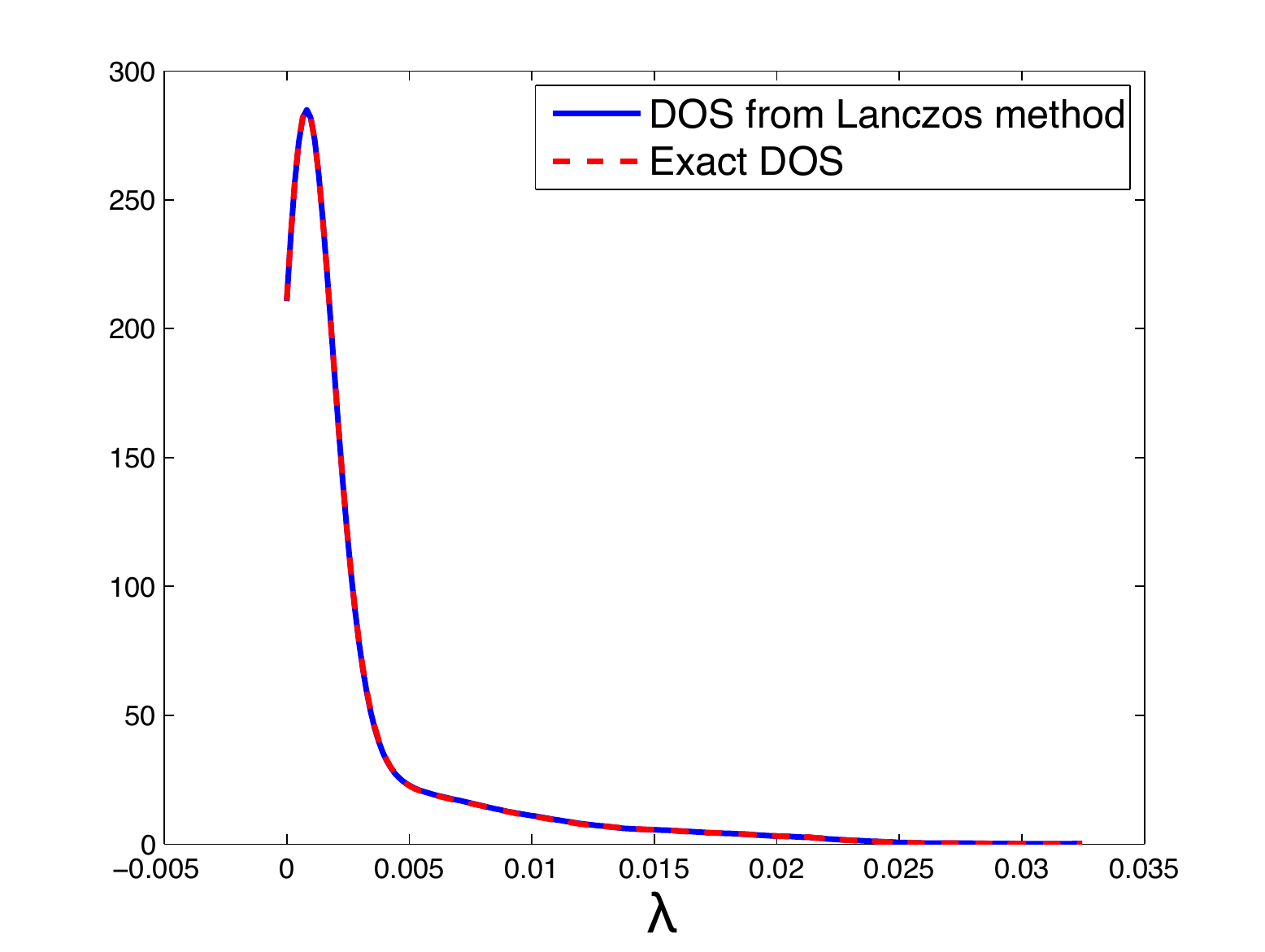}\\
{\small (iii) $\tau=10^{-3}$} & {\small (iv) $\tau=10^{-4}$}%
\end{tabular}
\caption{For the earth's normal mode matrix pencil, the computed DOS by the Lanczos method when $m = 30$ and $n_{nev} = 50$, compared to the exact DOS. The operations $B^{-1}v$ and $B^{-1/2}v$ are approximated by $g_{k_1}(B)v$ and $q_{k_2}(B)v$, respectively.  The degrees $k_1$ and $k_2$ are selected to be the smallest integers for which (\ref{eq:tau}) hold. The approximation errors are $1.41\!\times\!10^{-2}$, $5.61\!\times\!10^{-3}$, $4.70\!\times\!10^{-3}$ and $4.30\!\times\!10^{-3}$ when $\tau$ equals $10^{-1}$, $10^{-2}$, $10^{-3}$ and $10^{-4}$, respectively.}
\label{fig:lan2}
\end{figure}

\subsection{An example from a Tight-Binding calculation}
The second example is from  the Density Functional-based Tight Binding
(DFTB)  calculations (Downloaded from http://faculty.smu.edu/yzhou/data/matrices.htm). The  matrices  $A$  and   $B$  have  dimension
$n=17,493$. The matrix $A$ has  $3,927,777$ nonzero elements while $B$
has $3,926,405$  nonzero elements. The  eigenvalues of the  pencil are
ranging from $\lambda_{min} = -0.9138$ to $\lambda_{max} = 0.8238$.

Compared with  the earth's  normal mode matrix  pencil, both  $A,B$ in
this TFDB matrix pair are  much denser. Fig.~\ref{fig:pattern} displays the
sparsity patterns of $B$ and of its Cholesky factor, where $nz$ stands
for  the number  of  non-zeros. Even  with the  help  of AMD  ordering
\cite{AmestoyDavisDuff04},  the number  of non-zeros  in the  Cholesky
factor of  $B$ still reaches  $48,309,857$, which amounts  to having
$5.5233\!\times\!10^{3}$ non-zeros per row/column. This will cause two
issues.  First, a huge amount of memory may be needed to store the 
factors for a similar problem of
larger dimension.   Second,     applying  these   factors  is also   very
inefficient. These issues limit the  use of Cholesky factorization for
realistic large-scale calculations. On  the other hand, after diagonal
scaling  the  matrix  $B$  has eigenvalues  in  the  remarkably  tight
interval $[0.5756, \ 1.4432]$, which  allows a polynomial of degree as
low  as  $6$  for  $g_{k_1}$  and  $5$  for  $q_{k_2}$  when  $\tau  =
10^{-3}$. Thus,  we will  only test  the KPM  and Lanczos  method with
Chebyshev polynomial approximation techniques for this problem.

\begin{figure}[htb] 
\includegraphics[width=0.46\textwidth]{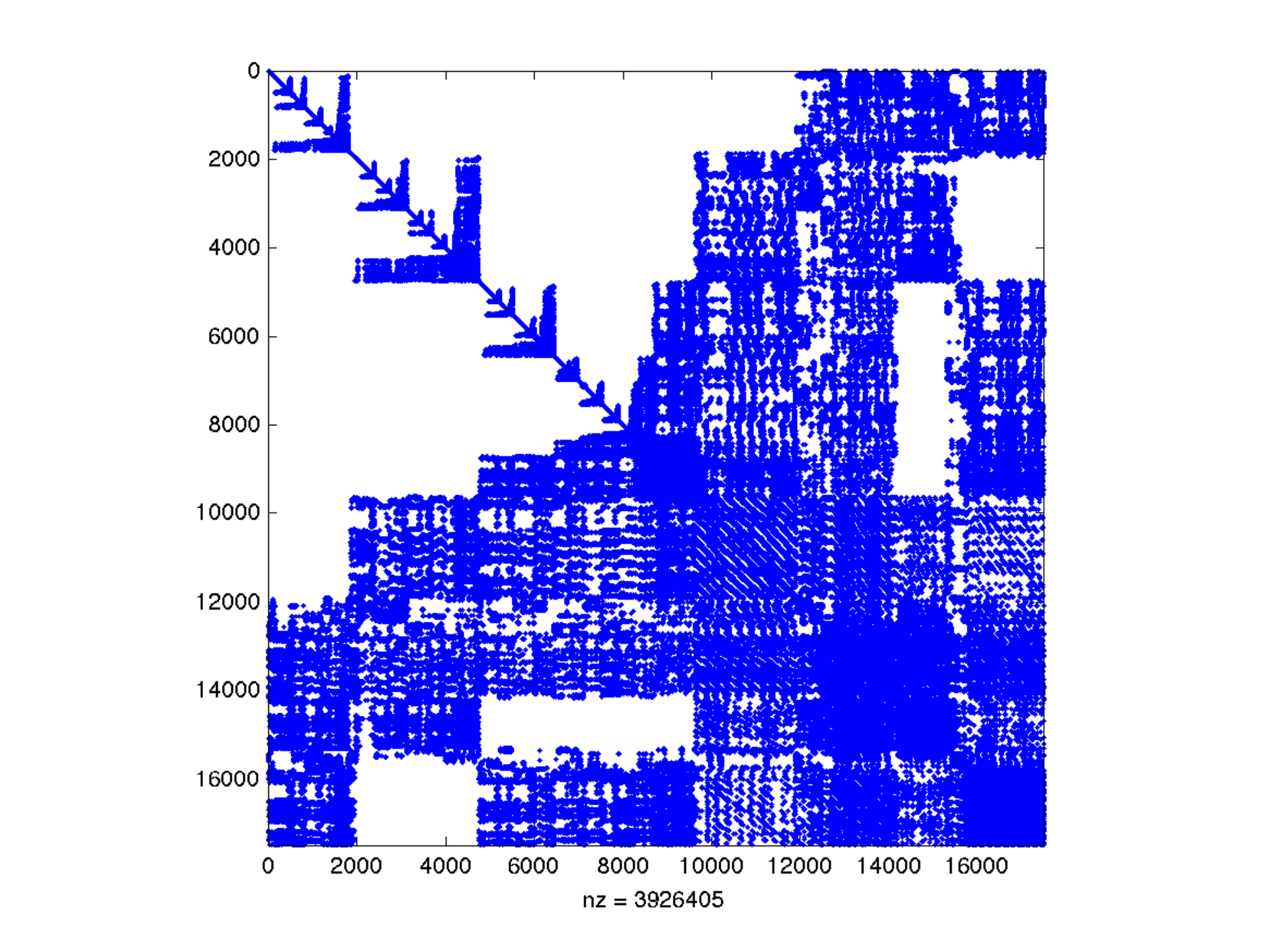}
\includegraphics[width=0.46\textwidth]{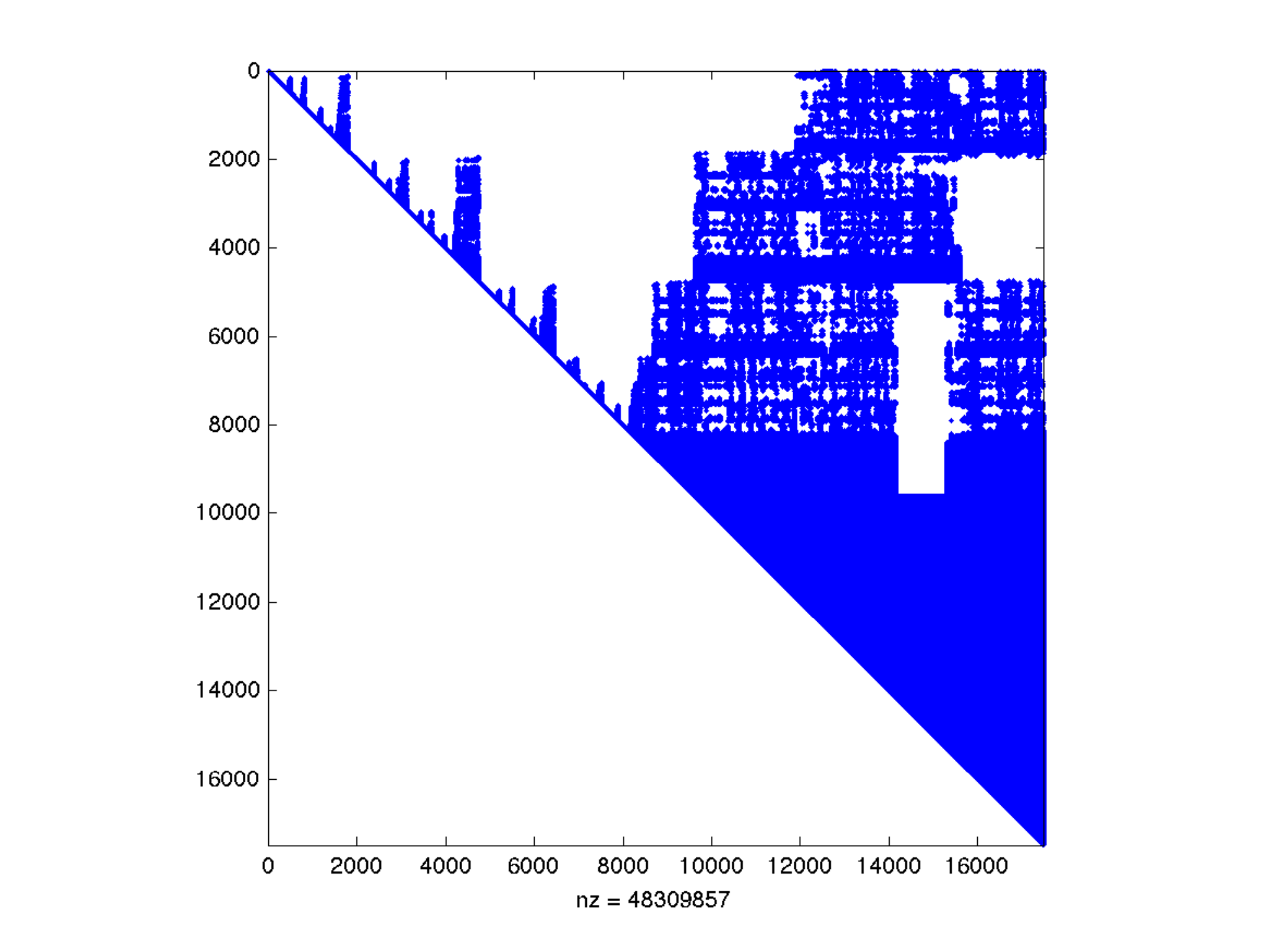}
\caption{The  sparsity  patterns of  the  matrix  $B$ (left)  and  its
  Cholesky factor (right) for the  TFDB matrix pencil. AMD ordering is
  applied to $B$ to reduce the number of non-zeros in its factors.}
\label{fig:pattern}
\end{figure}

In  the  experiment,   we  fixed  $m=30$  and   $n_{nev}=50$  in  both
methods. Fig.~\ref{fig:tfdb}  shows that  the quality of  the computed
DOS by  the KPM  method is clearly  not as good as the  one obtained  from the
Lanczos method. The error for the  KPM is $0.2734$ while the error for
the Lanczos method  is only $0.0058$. This is because  the spectrum of
$(A,B)$  has  four  heavy  clusters,  which  causes  difficulties  for
polynomial-based methods to capture the corresponding peaks on the DOS
curve.

\begin{figure}[htb] 
\includegraphics[width=0.46\textwidth]{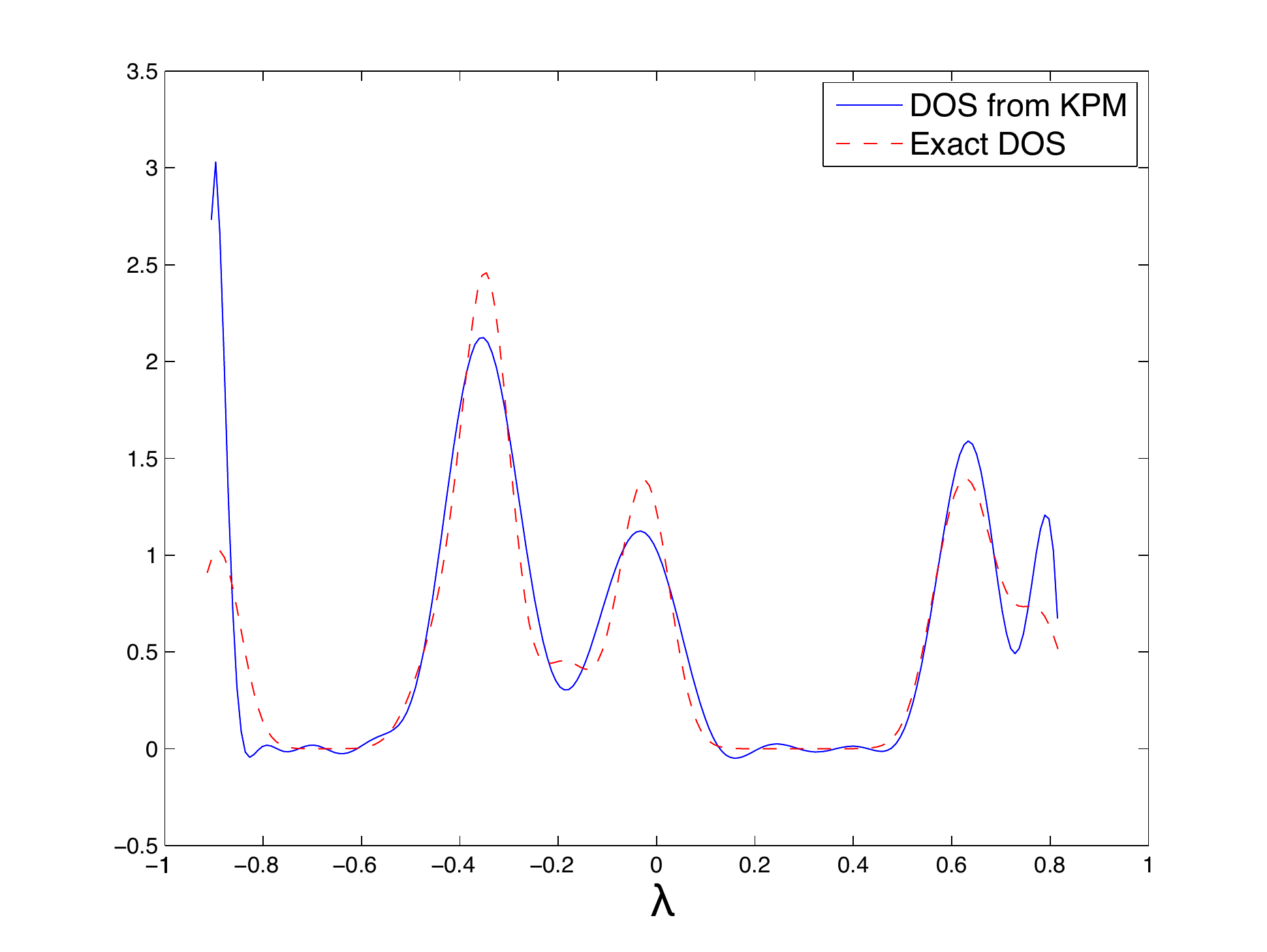}
\includegraphics[width=0.46\textwidth]{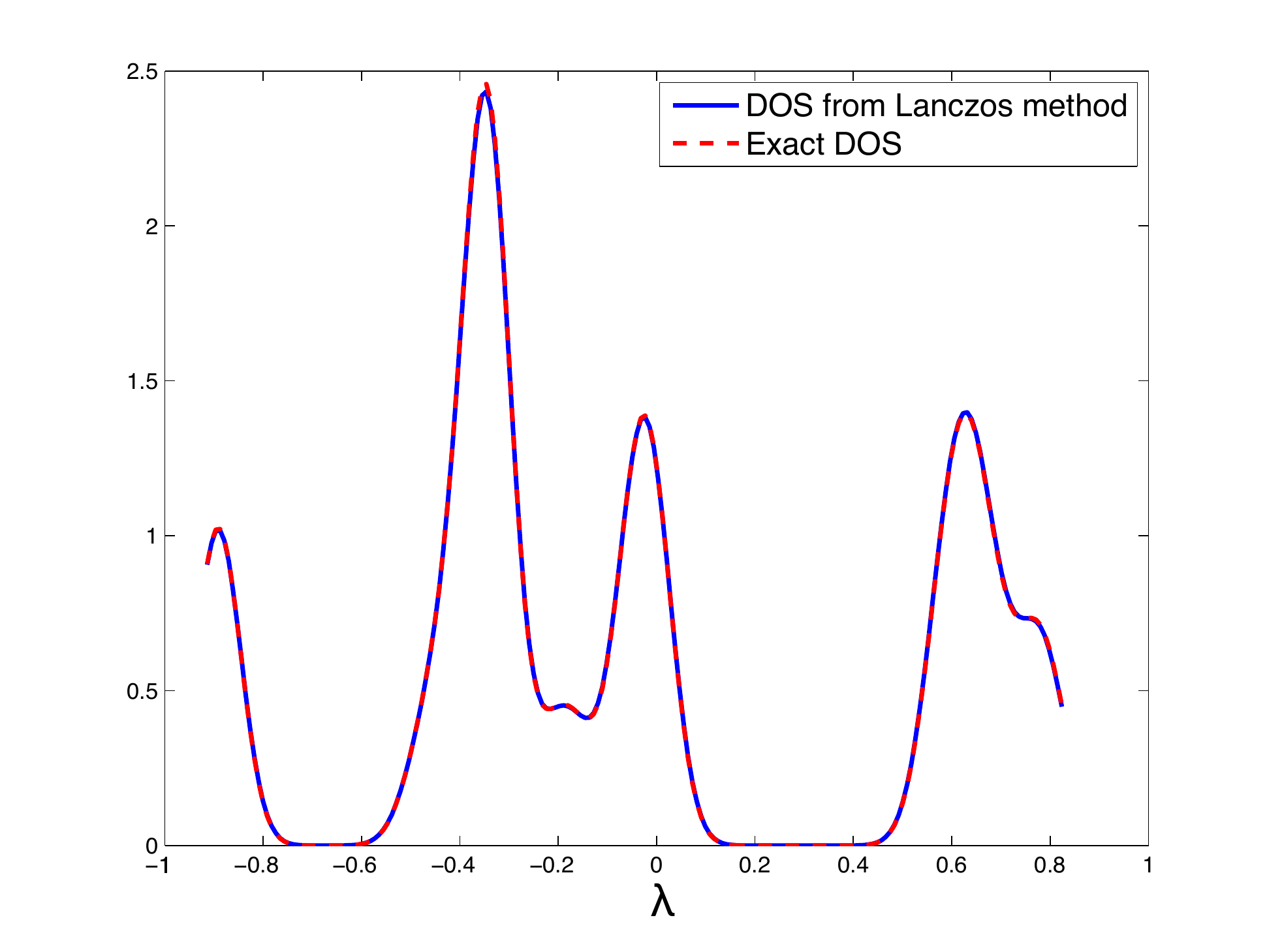}\\
\centerline{
\includegraphics[width=0.46\textwidth]{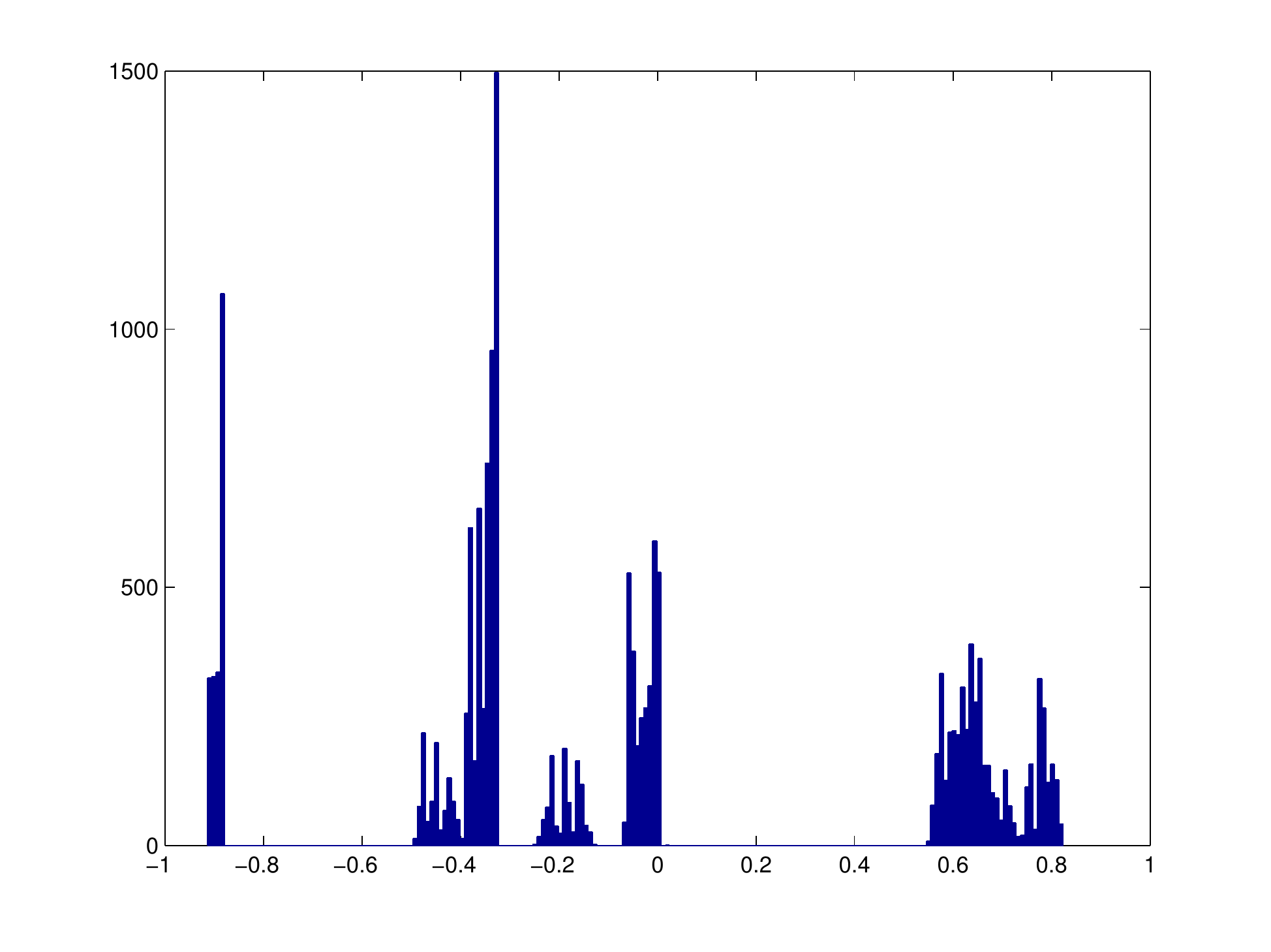}}
\caption{For the TFDB matrix pencil, the computed DOS by the KPM (upper left) 
and Lanczos method (upper right) when $m=30$ and $n_{nev}=50$, compared to the exact DOS and the histogram of the eigenvalues with $200$ bins (lower middle).
Chebyshev polynomial approximations are used for operations involving $B$.}
\label{fig:tfdb}
\end{figure} 

\subsection{Application: Slicing the spectrum}\label{sec:slicing}
 This  section discusses the  spectrum slicing
techniques implemented in the EVSL  package \cite{EVSL}. First,
the Lanczos method is invoked to get an approximate DOS $\tilde{\phi}$ of
the input  matrix pencil $(A,B)$: \eq{eq:LanDosest}  \tilde{\phi}(t) =
\frac{1}{s}    \sum_{k=1}^s   \sum_{i=1}^m    a_i\up{k}   g_{\sigma}(t
-{\theta_i\up{k}})    \    \    \text{with}\   \    g_{\sigma}(t)    =
\frac{1}{\sqrt{2\pi}\sigma}e^{\frac{-t^2}{2\sigma^2}}.    \en  Suppose
the users would  like to compute all the eigenvalues  located inside a
target interval  $[a,\ b]$  as well  as their  associated eigenvectors
with  $n_s$  slices. The  interval  $[a,\  b]$  will first  be  finely
discretized      with       $N+1$      evenly       spaced      points
$x_0=a\!<\!x_1\!<\!\ldots\!<\!x_{N-1}\!<\!x_{N}=b$,  followed  by  the
evaluation  of $\tilde{\phi}_i  :=  \tilde{\phi}(x_i)$  at each  point
$x_{i}$.

Then a numerical integration scheme is used to approximate the following integral based on the computed $\{\tilde{\phi}_i\}$
\[
y_i \approx \int_{a}^{x_{i}} \tilde{\phi}(t)dt.
\]
Each $y_i$ serves an approximation to the number of eigenvalues falling inside $[a,\ x_{i}]$. In particular, we know there are roughly $y_{N}$ eigenvalues inside $[a,\ b]$ and should expect an ideal partitioning yielding $y_{N}/n_s$ eigenvalues per slice.

The endpoints $\{t_i\}$ are identified as a subset of $x_{i}$. Start with $t_{0}=x_{0}$. The next $t_{i+1}$ for $i=0,\ldots, K-2$ is found by testing a sequence of $x_{j}$ starting with $t_{i} = x_{k}$ until  $y_{j} - y_k$ is approximately equal to $y_{K}/n_s$, yielding the point $t_{i+1} = x_{j}$. In the end, the points $\{t_{i}\}$ separate $[a,\ b]$ into $n_s$ slices.

We illustrate the  efficiency  of   this  slicing
mechanism with one example. The  test problem  is to  partition the  interval $[0.003,\
0.01]$  into   $5$  slices   for  the   earth's  normal   mode  matrix
pencil. Based  on Fig.~\ref{fig:lanvskpm2},  we know  that eigenvalues
are distributed unevenly within this interval. Therefore, a naive uniform
partitioning in which all sub-intervals have the same width   will cause  some
slices to contain many more  eigenvalues than others.   We fixed  $m$ at
$30$ and varied the number of sample vectors $n_{nev}$ to estimate the
DOS for  this matrix  pencil. The  resulting partitioning  results are
tabulated  in Table  \ref{tab:slicing}. As  we can  see, even  a small
number $n_{nev}$ can  still provide a reasonable  partitioning for the
purpose of balancing the memory usage associated with each slice.

\begin{table}[htb] 
\begin{center} 
\begin{tabular}{c|c|c|c|c|c|c} 
\multicolumn{1}{c|}{}&\multicolumn{2}{|c|}{$n_{nev}=10$}&\multicolumn{2}{|c|}{$n_{nev}=20$}&\multicolumn{2}{|c}{$n_{nev}=30$}\\\hline
$i$ & $[t_i,\ t_{i+1}] $ & $n_i$& $[t_i, \ t_{i+1}] $ & $n_i$& $[t_i, \ t_{i+1}] $ & $n_i$\\ \hline 
1 &  $[0.0030,\   0.0036] $ &  $84$&  $[0.0030,\   0.0036] $ &  $84$&$[0.0030,\   0.0036] $ &  $84$\\
2 &  $[0.0036,\   0.0045]$ &  $90$&$[0.0036,\   0.0045]$ &  $90$&  $[0.0036,\   0.0045]$ &  $90$\\
3 &  $[0.0045, \  0.0059]$ &  $105$&  $[0.0045, \  0.0059]$ &  $105$&  $[0.0045, \  0.0060]$ &  $113$\\
4 &  $[0.0059,\    0.0077]$ &  $113$ &  $[0.0059,\    0.0078]$ &  $119$ &  $[0.0060,\    0.0079]$ &  $115$\\
5 &  $[0.0077,\ 0.0100]$ &  $110$ &  $[0.0078,\ 0.0100]$ &  $104$&$[0.0079,\ 0.0100]$ &  $98$\\
\hline
\end{tabular}
\end{center} 
\caption{Partitioning $[0.003,\ 0.010]$ into $5$ slices $[t_i,\ t_{i+1}]$ for the earth's normal mode matrix pencil. The computational times for the Lanczos method are $0.53$s, $0.96$s and $1.58$s as the number of sample vectors $n_{nev}$ increases from $10$ to $30$. $n_i$ is the exact number of eigenvalues located inside the $i$th partitioned slice $[t_i,\ t_{i+1}]$.}\label{tab:slicing}
\end{table}

\section{Conclusion}  \label{sec:conclude}  
Algorithms that require only  matrix-vector
multiplications can offer enormous advantages over those that rely on 
factorizations. This has been observed for polynomial filtering techniques
for eigenvalue problems \cite{Filtlan-paper,spectrumslicing},
and it has also  just been illustrated in this paper which described two methods
to   estimate spectral densities  of  matrix pencils. These two
methods use  Chebyshev  polynomial  approximation techniques to
approximate  the  operations involving  $B$  and so they  only operate on  $(A,B)$
through matrix-vector multiplications. 

The  bounds that were established suggest  that the
Lanczos method may converge twice as fast  as the KPM method  under some
assumptions and it was confirmed experimentally to  produce more accurate estimation when the spectrum
contains clusters.
The proposed  methods are being implemented in C  in the EVSL package
\cite{EVSL} 
and will be   made available  in the  next release.

This study suggested that it is also possible to compute eigenvalues and vectors
of matrix pairs without any factorization. Theorem~\ref{thm:perturb} indicates  that 
 rough approximations of the eigenpairs can be  obtained 
by using a low-degree polynomial for $B^{-1/2}$. 
These approximations can be improved in a number of ways, e.g., by a Rayleigh-Ritz, or
a subspace iteration-type procedure. We plan on exploring this approach in our future work.

\bibliographystyle{siam}
%%\bibliography{paper}
\bibliography{dos.bib}

\end{document}

%% file: preamble.tex
\def\tr{\mbox{Trace} }
\def\half{{1\over2}}%

\def\inv{^{-1}}%
\def\invt{^{-T}}%
\def\nref#1{(\ref{#1})}

\def\comb#1,#2,{ \left( {#1 \atop #2 } \right)  }%

\def\Span{\mbox{Span}}

\newtheorem{algor}{{\sc Algorithm}}[section]

\def\betab{\vspace{0.001in}\begin{tabbing} 
xxxx\=xxxx\=xxxx\=xxxx\=xxxxxx\=xxxxxxxxxxxxxxxxxxxxxxxxxxxxx\=xx\=xx\=xx\=xx\=xx\=xx\=xx\= \kill}
\def\entab{\end{tabbing}\vspace{0.15in}}

\newcommand{\eq}[1]{\begin{equation}\label{#1}}
\newcommand{\en}{\end{equation}}